\documentclass[12pt,letterpaper]{amsart}
\usepackage{amssymb,amsmath}
\usepackage[margin=1in]{geometry}

\title[{K}lein-{G}ordon on asymptotically de {S}itter spaces]{A
  parametrix for the fundamental solution of the {K}lein-{G}ordon
  equation on asymptotically de {S}itter spaces}
\author{Dean Baskin}
\address{Department of Mathematics, Stanford University, Stanford CA
  94305} 
\date{May 31, 2010}

\usepackage{xypic}
\usepackage{tikz}

\definecolor{StanfordRed}{rgb}{0.6431,0,0.1137}

\newtheorem{thm}{Theorem}
\newtheorem{lem}[thm]{Lemma}
\newtheorem{prop}[thm]{Proposition}

\theoremstyle{definition}
\newtheorem{defn}[thm]{Definition}
\theoremstyle{remark}
\newtheorem{note}[thm]{Remark}

\newcommand{\azero}{{}^{0}\alpha}
\newcommand{\asympto}{\sim}
\newcommand{\charset}{\Sigma}
\newcommand{\complexes}{\mathbb{C}}
\newcommand{\diag}{\Delta}
\newcommand{\Diff}[1][]{\operatorname{Diff}^{#1}}
\newcommand{\extcup}{\bar{\cup}}
\newcommand{\E}{\mathcal{E}}
\newcommand{\F}{\mathcal{F}}
\newcommand{\G}{\mathcal{G}}
\newcommand{\hd}{\Omega^{\frac{1}{2}}}
\newcommand{\Hzero}{{}^{0}H}
\newcommand{\K}{\mathcal{K}}
\newcommand{\integers}{\mathbb{Z}}
\newcommand{\liftsto}{\leadsto}
\newcommand{\naturals}{\mathbb{N}_{0}}
\newcommand{\norm}[2][]{\left\| #2 \right\| _{#1}}
\newcommand{\lap}{\Delta}
\newcommand{\pd}[1][]{\partial_{#1}}
\newcommand{\reals}{\mathbb{R}}

\newcommand{\Tzero}{{}^{0}T}
\newcommand{\WF}{\operatorname{WF}}
\newcommand{\wdbl}{\tilde{\omega}}
\newcommand{\wzero}{{}^{0}\omega}

\newcommand{\CDf}[4][]{I^{#2}_{#1}(#3 ;#4)}
\newcommand{\LDf}[5][]{I^{#2}_{#1}(#3 ; #4)}
\newcommand{\LDflong}[5][]{I^{#2}_{#1}(#3 ; #4 ; #5)}
\newcommand{\phgalt}[1]{\mathcal{A}^{#1}_{\text{phg}}}
\newcommand{\phg}[3]{\mathcal{A}_{\text{phg}}^{#1}(#2 ; #3)}
\newcommand{\PLflong}[6][]{I_{#1}^{#2}(#3 ; #4 , #5; #6)}
\newcommand{\PLf}[6][]{I_{#1}^{#2}(#3 ; #4 , #5)}
\newcommand{\PLzero}[1][m]{\PLf[0]{#1}{\dblzero}{\Lambda_{0}}{\Lambda_{1}}{\hd (\dblzero)}}
\newcommand{\PLzerolong}[1][m]{\PLflong[0]{#1}{\dblzero}{\Lambda_{0}}{\Lambda_{1}}{\hd (\dblzero)}}
\newcommand{\PLlong}[1][m]{\PLflong{#1}{\dblzero}{\Lambda_{0}}{\Lambda_{1}}{\hd (\dblzero)}}

\newcommand{\st}{\tilde{s}}
\newcommand{\xt}{\tilde{x}}
\newcommand{\yt}{\tilde{y}}
\newcommand{\zt}{\tilde{z}}

\newcommand{\differential}[1]{\,d#1}
\newcommand{\dg}{\differential{\hat{g}}}
\newcommand{\dr}{\differential{r}}
\newcommand{\ds}{\differential{s}}
\newcommand{\dt}{\differential{t}}
\newcommand{\dT}{\differential{T}}
\newcommand{\dw}{\differential{w}}
\newcommand{\dx}{\differential{x}}
\newcommand{\dxt}{\differential{\xt}}
\newcommand{\dy}{\differential{y}}
\newcommand{\dyt}{\differential{\yt}}
\newcommand{\dz}{\differential{z}}
\newcommand{\deta}{\differential{\eta}}
\newcommand{\dmu}{\differential{\mu}}
\newcommand{\dtau}{\differential{\tau}}
\newcommand{\dtheta}{\differential{\theta}}
\newcommand{\dxi}{\differential{\xi}}
\newcommand{\domega}{\differential{\omega}}

\newcommand{\dblzero}{X^{2}_{0}}
\newcommand{\dblspacet}{X^{2}_{0,t}}
\newcommand{\dblspace}{\widetilde{X}^{2}_{0}}
\newcommand{\frontface}{\operatorname{ff}}
\newcommand{\ffp}{\frontface_{+}}
\newcommand{\ffm}{\frontface_{-}}
\newcommand{\leftface}{\operatorname{lf}}
\newcommand{\lfp}{\leftface_{+}}
\newcommand{\lfm}{\leftface_{-}}
\newcommand{\rightface}{\operatorname{rf}}
\newcommand{\rfp}{\rightface_{+}}
\newcommand{\rfm}{\rightface_{-}}
\newcommand{\newface}{\operatorname{lcf}}
\newcommand{\lcfp}{\newface_{+}}
\newcommand{\lcfm}{\newface_{-}}
\newcommand{\xf}{\operatorname{scf}}
\newcommand{\lightcone}{\operatorname{LC}} 

\begin{document}


\begin{abstract}
  In this paper we construct a parametrix for the forward fundamental
  solution of the wave and Klein-Gordon equations on asymptotically de
  Sitter spaces without caustics. We use this parametrix to obtain
  asymptotic expansions for solutions of $(\Box - \lambda) u = f$ and
  to obtain a uniform $L^{p}$ estimate for a family of bump functions
  traveling to infinity. 
\end{abstract}

\maketitle

\section{Introduction}
\label{sec:introduction}

De Sitter space is an exact solution of the vacuum Einstein
equations with positive cosmological constant. In this paper, we study
the forward fundamental solution of the wave and Klein-Gordon
equations on asymptotically de Sitter spaces. This is the unique
operator $E_{+}$ (which we identify with its Schwartz kernel) which
satisfies $(\Box - \lambda ) E_{+} = I$ and is supported in the
forward light cones, i.e., for a compactly supported smooth function
$f$, the function $u = E_{+}f$ satisfies 
\begin{align}
  \label{eq:inhomog-eqn}
  &(\Box - \lambda ) u = f \\
  &u \equiv 0 \text{ near past infinity}. \notag
\end{align}
Here $\lambda$ is the Klein-Gordon parameter. If
equation~\eqref{eq:inhomog-eqn} is considered as a
massive wave equation, the condition $\lambda \leq 0$ corresponds
to positive mass. We construct a parametrix for this problem and
establish asymptotic expansions for solutions. As an application of
our parametrix, we prove a uniform $L^{p}$ estimate for the operator
applied to a family of bump functions tending toward infinity. We
postpone to a future paper the consideration of Strichartz estimates
and semilinear wave equations on asymptotically de Sitter spaces.

The study of the decay properties of the wave equation on various
natural classes of spacetimes is an active area of research. For
example, Dafermos and Rodnianski \cite{dafermos-rodnianski:2007:long}
and Melrose, S{\'a} Barreto, and Vasy
\cite{melrose-sabarreto-vasy:2008} have obtained decay results for
solutions of the wave equation in the context of the de
Sitter-Schwarzschild model of a black hole spacetime.

Our definition of asymptotically de Sitter spaces is given in
\cite{vasy:2007} (called asymptotically de Sitter-like spaces in that
manuscript) and follows the definition of asymptotically hyperbolic
spaces given in \cite{mazzeo:1988} and \cite{mazzeo-melrose:1987}. An
asymptotically de Sitter space is a compact manifold with boundary
equipped with a Lorentzian metric having a prescribed asymptotic form
near the boundary. This pushes the boundary off ``to infinity''.

The microlocal structure of the fundamental solution for general real
principal type operators has been studied extensively. The solution
operator for the Cauchy problem for general real principal type
operators is a Fourier integral operator associated to a Lagrangian
submanifold of phase space given by the flowout of the Hamilton vector
field of the principal symbol of the operator. This was first
described in this language by Duistermaat and H{\"o}rmander in
\cite{duistermaat-hormander:1972}. In \cite{melrose-uhlmann:1979},
Melrose and Uhlmann constructed the forward fundamental solution for a
real principal type operator as a paired Lagrangian distribution.
These are distributions associated to two cleanly intersecting
Lagrangian submanifolds in phase
space. 

Guillemin and Uhlmann \cite{guillemin-uhlmann:1981}, Joshi
\cite{joshi:1998}, Melrose and Zworski \cite{melrose-zworski:1996},
Hassell and Vasy \cite{hassell-vasy:1999}, and others have all
generalized the notion of paired Lagrangian distributions. Guillemin
and Uhlmann defined a much more general class of paired Lagrangian
distributions, which Joshi restricted slightly in order to construct a
well-behaved calculus. Joshi then used this calculus to construct
complex powers of the wave operator on Riemannian manifolds. Melrose
and Zworski defined a class of distributions associated to an
intersecting pair of Legendrians, while Hassell and Vasy later
expanded this notion to describe the spectral projections on a
scattering manifold.

Polarski~\cite{polarski:1989} computed the propagator for the equation
of the massless conformally coupled scalar field in the static de
Sitter metric, which has been transformed to the Einstein open
universe. Yagdjian and Galstian~\cite{yagdjian-galstian:2008} computed
the fundamental solutions for the Klein-Gordon equation in de Sitter
spacetime transformed by the Lema\^{i}tre-Robertson change of
coordinates to the special case of the
Friedmann-Robertson-Walker-Lema\^{i}tre spacetime. They represented
the fundamental solutions and solutions of the Cauchy problem using
hypergeometric functions and proved $L^{p}L^{q}$ estimates.

Vasy~\cite{vasy:2007} generalized and extended Polarski's result to
asymptotically de Sitter spaces. He exhibited the well-posedness of
the Cauchy problem and showed that on such spaces, the solution $u$ of
$(\Box - \lambda) u = 0$ with smooth Cauchy data has an asymptotic
expansion at infinity. Indeed, if $x$ is a boundary defining function
for the conformal compactification of an asymptotically de Sitter
space $X$ and $\sqrt{\frac{(n-1)^{2}}{4} + \lambda}$ is not a
half-integer, then $u$ has an expansion 
\begin{equation*}
  u = u_{+}x^{\frac{n-1}{2}+\sqrt{\frac{(n-1)^{2}}{4}+\lambda}} + u_{-}x^{\frac{n-1}{2}-\sqrt{\frac{(n-1)^{2}}{4}+\lambda}},
\end{equation*}
where $u_{+}$ and $u_{-}$ are smooth on $X$. (The difference between this
expression and the corresponding one in \cite{vasy:2007} is due to
differing sign conventions for Lorentzian metrics.) In the case of an
integer coincidence, $u_{-}$ is instead in $C^{\infty}(X) +
x^{2\sqrt{\frac{(n-1)^{2}}{4}+\lambda}}(\log x)C^{\infty}(X)$. Vasy
also showed that solutions of the wave equation exhibit scattering,
i.e., that the data $u_{\pm}$ may be specified at one of $Y_{\pm}$
(future and past infinity, respectively), which fixes the data at
$Y_{\mp}$.

Our result extends the work of Vasy to the study of the fundamental
solution $(\Box - \lambda)E_{+} = I$ on asymptotically de Sitter
spaces.  We require three global assumptions in our study of
asymptotically de Sitter spaces: 
\begin{enumerate}
\item [(A1)] $Y = Y_{+}\cup Y_{-}$, with $Y_{+}$ and $Y_{-}$ a union
  of connected components of $Y$, 
\item [(A2)] each bicharacteristic $\gamma$ of $P$ converges to
  $Y_{+}$ as $t \to +\infty$ and to $Y_{-}$ as $t\to -\infty$, or vice
  versa, and 
\item [(A3)] The projection from $T^{*}X$ to $X$ of the flowout of the
  forward light cone from any point $p\in Y_{-}$ is an embedded
  submanifold of $X$ (except at the point $p$, where it always has a
  conic singularity). 
\end{enumerate}
The first two of these assumptions are not particularly restrictive.
They imply that the manifold is topologically a product $Y_{+}\times
\reals$, but are reasonable from a physical viewpoint in that they
imply a time orientation on the manifold. Colloquially, assumptions
(A1) and (A2) prevent the breakdown of causality on $X$.

Assumption (A3) is needed only to obtain sharp global statements about
the fundamental solution, but our construction works in some
neighborhood of $Y_{+}$ (i.e., a neighborhood of future infinity)
without this assumption.

Assumption (A3) prohibits the development of caustics, which
significantly narrows the class of manifolds considered. De Sitter
space just misses being covered by (A3), though a slight modification
of our construction still applies here. Indeed, the projection of the
flowout of the light cone from a point $p\in Y_{-}$ is a smooth
embedded submanifold of the interior of de Sitter space, but
intersects itself at $Y_{+}$. Section~\ref{sec:exact-de-sitter}
discusses the minor modifications needed to handle this case.

If we slightly enlarge the spherical cross section of de Sitter space,
then this new space satisfies the assumptions above. The assumptions
above are stable under perturbation, so the construction applies to
perturbations of this enlarged version of de Sitter space.

In order to remove the assumption (A3), we could combine the Poisson
operator construction of \cite{vasy:2007} with the local description
of the forward fundamental solution (given in
\cite{melrose-uhlmann:1979}). Phrasing this in a geometric way is left
to a future paper.

The main result of this paper is the following (we state it more
precisely and define the relevant classes of distributions later):
\begin{thm}
  \label{thm:main-thm}
  Suppose $(X,g)$ is an asymptotically de Sitter space.  There is a
  compactification $\dblspace$ of the interior of $X\times X$ to a
  compact manifold with corners such that the closures of the diagonal
  and the light cone both intersect all boundary hypersurfaces
  transversely. $\dblspace$ is constructed by first blowing up the
  boundary of the diagonal in $X\times X$ and then blowing up the set
  where the projection of the flowout of the light cone hits the side
  face. On this compactification, the forward fundamental solution of
  $\Box - \lambda$ lifts to be the sum of a paired Lagrangian
  distribution, smooth down to the front face, and a conormal
  distribution associated to the light cone with polyhomogeneous
  expansions at the other faces of $\dblspace$.
\end{thm}

We call a tempered distribution $f$ on $X$ \emph{forward-directed} if it
a smooth function on the interior of $X$ that vanishes to all orders
at $Y_{-}$. The work of Vasy \cite{vasy:2007} implies that if $f$ is
forward-directed, then we may apply $E_{+}$ to $f$.

A function $f$ is \emph{polyhomogeneous} with index set $E$ on $X$ if it has
an asymptotic expansion of the form 
\begin{equation*}
  \sum _{(r,l)\in E} x^{r}(\log x)^{l}a_{rl}(y)
\end{equation*}
near $Y_{\pm}$.  

Let $F_{1}$ and $F_{2}$ be the following two index sets:
\begin{align}
  \label{eq:index-sets}
  F_{1} &= \left\{ (j, l) : j,l\in \naturals, l\leq j\right\}, \notag \\
  F_{2} &= \{ (s_{\pm}(\lambda) + m , 0 ) : m\in \naturals  \} ,
\end{align}
where $s_{\pm}(\lambda) = \frac{n-1}{2} \pm \sqrt{\frac{(n-1)^{2}}{4}
  + \lambda}$.

A corollary of Theorem~\ref{thm:main-thm} describes the
polyhomogeneity of the solutions of $P(\lambda) u = f$. 
\begin{thm}
  \label{thm:mapping-phg}
  If $f$ is forward-directed and polyhomogeneous on $X$ with index set
  $E$ at $Y_{+}$, then $E_{+}f$ is forward-directed and
  polyhomogeneous with index set $F$, where
  \begin{equation*}
    F = F _{1}\,\extcup \,F _{2} \,\extcup \,E .
  \end{equation*}
\end{thm}
Here $E_{+}$ is regarded as an operator and $\extcup$ denotes the
extended union of two index sets:
\begin{equation*}
  (z,p) \in E \extcup F \Leftrightarrow (z,p)\in E\cup F \text{ or } p = p' + p'' +1 \text{ with } (z,p')\in E\text{ and }(z,p'')\in F .
\end{equation*}

As an application of our parametrix, we obtain a uniform $L^{p}$
estimate for a family of bump functions traveling to infinity. We
leave more general $L^{p}$ mapping properties and Strichartz estimates
to a future manuscript.

\begin{thm}
  \label{thm:Lp-estimates}
  Suppose $\phi \in C^{\infty}_{c}(\reals_{s}^{+}\times
  \reals_{z}^{n-1})$ is supported near $(1,0)$. For $(\xt, \yt) \in
  X$, Let $f_{(\xt, \yt)}(x,y) = \phi (\frac{x}{\xt},
  \frac{y-\yt}{\xt})$. Suppose that $p,l$, and $r$ satisfy
  \begin{align*}
    &2 < p < \infty, \\
    &r > \max \left( \frac{1}{2}, \Re \sqrt{\frac{(n-1)^{2}}{4} +
        \lambda}\right), \quad \text{and}\\
    &l \geq \max\left( 0,  - \frac{n-1}{2} + \Re \sqrt{\frac{(n-1)^{2}}{4} + \lambda}\right).
  \end{align*}
  Then   
  \begin{equation}
    \label{eq:Lp-est-main}
    \norm[x^{- l + 2(l-r)/p}L^{p}(X;\differential{g})]{E_{+}f_{(\xt, \yt)}}\leq C\xt ^{-2r/p}.
  \end{equation}
  Here the constant $C$ depends on the family $f_{(\xt,\yt)}$ but not
  on $(\xt,\yt)$.

  In particular, for the wave equation, $\lambda = 0$, and we may let
  $l=0$ and $r > \max \left( \frac{1}{2}, \frac{n-1}{2}\right)$.   In
  this case,
  \begin{equation*}
    \norm[x^{-2r/p}L^{p}(X;\differential{g})]{E_{+}f_{(\xt, \yt)}}\leq C\xt ^{-2r/p}.
  \end{equation*}
\end{thm}

Theorem~\ref{thm:main-thm} follows from a parametrix construction that
comprises the bulk of this paper. The main ingredient in the proof of
Theorem~\ref{thm:mapping-phg} is the pushforward theorem of Melrose.
Theorem~\ref{thm:Lp-estimates} follows from Theorem~\ref{thm:main-thm}
and the $L^{2}$ estimates proved by Vasy~\cite{vasy:2007}.

The paper is broadly divided into three parts. The first part consists
of Sections~\ref{sec:asympt-de-sitt} through \ref{sec:another-blow-up}
and develops the tools necessary to construct the parametrix. The
second part contains the construction, which begins with an outline in
Section~\ref{sec:outline-construction} and concludes in
Section~\ref{sec:full-parametrix}. The last part, consisting of
Sections~\ref{sec:mapping-props-phg} and \ref{sec:lp-estimate}, proves
Theorems~\ref{thm:mapping-phg} and \ref{thm:Lp-estimates}.

\section{Asymptotically de {S}itter spaces}
\label{sec:asympt-de-sitt}

We start by describing the de Sitter space. Recall that hyperbolic
space can be realized as one sheet of the two-sheeted hyperboloid in
Minkowski space. It inherits a Riemannian metric from the Lorentzian
metric in Minkowski space. De Sitter space, on the other hand, is the
one-sheeted hyperboloid $\{ -X_{0}^{2} + \sum _{i=1}^{n}X_{i}^{2} =
1\}$ in Minkowski space, but now the induced metric is Lorentzian. A
good set of coordinates on this space, which is diffeomorphic to
$S^{n}\times \reals$, is 
\begin{align*}
  X_{0} &= \sinh \tau \\
  X_{i} &= \omega _{i} \cosh \tau , 
\end{align*} 
where $\omega _{i}$ are coordinates on the unit sphere. The de Sitter
metric is then 
\begin{equation*}
  -\dtau ^{2} + \cosh^{2}\tau \domega^{2} .
\end{equation*}
If we let $T = e^{-\tau}$ near $\tau = +\infty$, then $\tau = +\infty$
corresponds to $T=0$ and the metric now has the form 
\begin{equation}
  \label{eq:de-sitter-metric}
  \frac{-\dT^{2} + \frac{1}{4}(T^{2}+1)^{2}\domega^{2}}{T^{2}}.
\end{equation}

Our definition of an asymptotically de Sitter space is based on the
form (\ref{eq:de-sitter-metric}) of the de Sitter metric.

\begin{defn}
  \label{def:asymp-de-sitter}
  $(X,g)$ is \emph{asymptotically de Sitter} if $X$ is an
  $n$-dimensional compact manifold with boundary $Y$, $g$ is a
  Lorentzian metric on the interior of $X$, and, for a boundary
  defining function $x$, there is some collar neighborhood
  $[0,\epsilon )_{x}\times Y$ of the boundary on which $g$ has the
  form
  \begin{equation*}
    g = -\frac{dx^{2}}{x^{2}} + \frac{h(x,y,dy)}{x^{2}},
  \end{equation*}
  where $h$ is a smooth symmetric $(0,2)$-tensor on $Y$ that is a
  metric on the boundary $\{ 0\} \times Y$. 
\end{defn}

A calculation similar to the ones in \cite{mazzeo:1988} or
\cite{mazzeo-melrose:1987} shows that the sectional curvatures of
asymptotically de Sitter spaces approach $1$ as $x\to 0$.

Let $\Box$, the wave operator, be the wave operator (d'Alembertian) associated to $g$,
and, for $\lambda \in \reals$, let $P = P(\lambda) = \Box - \lambda$
be the Klein-Gordon operator. We are seeking the fundamental solution
of the Klein-Gordon equation, i.e., a distribution $E_{+,\lambda}$ such
that $P(\lambda)E_{+,\lambda} = I$. Note that the principal symbol
$\sigma _{2}(P)$ is given by the dual metric function of $g$.

\begin{defn}
  \label{def:bichars}
  We say that the \emph{bicharacteristics} of $P$ over $X^{\circ}$,
  the interior of $X$, are the integral curves of the Hamilton vector
  field of the principal symbol $\sigma_{2}(P)$ inside the
  characteristic set of $P$ (the set where $\sigma _{2}(P)$ vanishes).
\end{defn}

Throughout this paper, we assume that both (A1) and (A2) hold. These
assumptions are introduced primarily to ensure that $X$ exhibits a
coherent causal structure.

Because $g$ is conformal to the incomplete pseudo-Riemannian metric
\begin{equation*}
  \hat{g} = -dx^{2} + h
\end{equation*}
near $Y$, the bicharacteristics of $P$ (near $Y$) are
reparametrizations of the bicharacteristics of $-dx^{2} + h$. $g$ is
complete, so the global assumptions imply that each bicharacteristic
$\gamma (t)$ has a limit in $S^{*}_{Y}X$ as $t\to \pm \infty$.

Physicists (e.g.\ \cite{geroch:1970}) have long known that these
assumptions imply the existence of a global `time' function $T\in
C^{\infty}(\overline{X})$ such that $T|_{Y_{\pm}}=\pm 1$ and is
monotone on the nullbicharacteristics of $P$. Note that $1-x$ and
$x-1$ have the desired properties near $Y_{+}$ and $Y_{-}$, so the
assumptions mean that such functions can be extended to all of $X$.
Moreover, $T$ gives a fibration $\overline{X}\to [-1,1]$ and so $X$ is
diffeomorphic to $X\times S$ for a compact manifold $S$. This also
shows that $Y_{+}$ and $Y_{-}$ are both diffeomorphic to $S$.

We also assume that assumption (A3) holds. This assumption ensures
that the boundary of the flowout of the light cone in the cotangent
bundle is actually the cotangent bundle of a submanifold in the base.
It allows us to work with an adapted class of conormal distributions
rather than a class of Lagrangian distributions. Note that the
projection of the flowout is automatically an embedded submanifold for
small times. Moreover, because the nullbicharacteristics agree in the
compact and non-compact settings (i.e., for $\hat{g}$ and $g$), there
is a neighborhood of $Y_{+}$ such that the projection of the flowout
of the forward light cone from any point in this neighborhood is an
embedded submanifold. If we then restrict to data supported in this
neighborhood, we may use our construction below even in the absence of
assumption (A3).

In this paper we adopt the convention that $\int \delta (x-\xt) \delta
(y-\yt)f(\xt, \yt)\dg = f(x,y)$. In particular, this breaks the formal
self-adjointness of the operator $P(\lambda)$. When we seek a right
parametrix, then, we must use the transpose of the operator
$P(\lambda)$ with respect to this metric. We may write this out
explicitly as \begin{equation}
  \label{eq:op-transpose}
  P^{t}(\lambda) = (x\pd[x])^{2} + (n+1)x\pd[x] + \frac{x\pd[x]\sqrt{h}}{\sqrt{h}}x\pd[x] + x^{2} \lap _{y} + n(1 + \frac{x\pd[x]\sqrt{h}}{\sqrt{h}}) - \lambda .
\end{equation}

We also adopt the convention that when we are applying a differential
operator to the variables in the right factor, we use a subscript $R$.

In our consideration of the forward fundamental solution of
$P(\lambda)$, it is useful to have the notion of
\emph{forward-directed} data, which are simply tempered distributions
(or polynomially bounded smooth functions) on $X$ that vanish to all
orders at $Y_{-}$.

\section{The {C}auchy problem and scattering}
\label{sec:existence}

In \cite{vasy:2007}, Vasy studied both the Cauchy problem and the
scattering problem on asymptotically de Sitter spaces. The Cauchy
problem seeks a solution the equation 
\begin{equation}
  \label{eq:Vasy-Cauchy-prob}
  \begin{cases}
    P(\lambda) u = 0, \\
    u |_{\Sigma_{0}} = \psi_{0}, & V u |_{\Sigma_{0}} = \psi _{1}
  \end{cases},
\end{equation}
where $\Sigma_{0}$ is a Cauchy hypersurface, $V$ is a vector field
transverse to $\Sigma_{0}$, and $\psi_{0}, \psi_{1}$ are smooth
functions on $\Sigma_{0}$.

Via a positive commutator argument, Vasy showed global existence and
uniqueness for solutions of the Cauchy
problem~(\ref{eq:Vasy-Cauchy-prob}) on asymptotically de Sitter
spaces. One may use any ``time slice'' $\{ T=\text{const}\}$ from the
diffeomorphism $X \cong \reals \times Y_{+}$ as the Cauchy
hypersurface needed to pose the Cauchy problem.

Vasy further proved the following a global solvability result for the
inhomogeneous equation on asymptotically de Sitter spaces under
assumptions (A1) and (A2). 
\begin{thm}[Theorem 5.4 of \cite{vasy:2007}]
  \label{thm:Vasy-thm-5.4}
  Suppose that $\lambda\in \reals$, and that $l_{\pm}$ satisfy
  \begin{equation*}
    l_{+} > \max (\frac{1}{2}, l(\lambda)), l_{-} < -\max(\frac{1}{2} , l(\lambda)).
  \end{equation*}
  Then for $f\in H^{0,l_{+},l_{-}}_{0}(X)$, $Pu = f$ has a unique
  solution $u\in H^{1,l_{+},l_{-}}_{0}(X)$. 
\end{thm}

Here $l(\lambda) = \Re\sqrt{\frac{(n-1)^{2}}{4} + \lambda}$ while the
$H^{m,q_{+},q_{-}}_{0} = x_{+}^{q_{+}}x_{-}^{q_{-}}H^{m}_{0}$ measure
both regularity and decay at $Y_{+}$,$Y_{-}$ separately. The spaces
$H^{m}_{0}$ are weighted Sobolev spaces measuring regularity with
respect to the $0$-vector fields defined in the following section.

In particular, for any $f\in \dot{C}^{\infty}(X)$ (indeed, for any
forward-directed tempered smooth function), there is a unique solution
$u\in x_{+}^{-\infty}x_{-}^{\infty}H^{1}_{0}(X)$ to $Pu=f$. (Here
$x_{\pm}$ is a defining function for $Y_{\pm}$, so the notation means
that $u$ is tempered at $Y_{+}$ and vanishes to infinite order at
$Y_{-}$.) There is thus a distribution $E_{+}$ on $X\times X$ that can
be called the forward fundamental solution of $P$. In other words,
$E_{+}$ is such that for each $p\in X$, $PE_{+}(p) = \delta _{p}$ and
$E_{+}(p)(q) \equiv 0$ for $q\in X$ not in the domain of influence of
$p$. In terms of the identification $X \cong \reals_{t} \times Y_{+}$,
this means that $E_{+}(p)(q)$ vanishes when $t(q) < t(p)$.

Vasy (cf.\ Theorem 5.5 of \cite{vasy:2007}) proved also that solutions
of the Cauchy problem~(\ref{eq:Vasy-Cauchy-prob}) have unique
asymptotic expansions near the boundary $Y$ of $X$ of the form
\begin{equation*}
  u = u_{+}x^{\frac{n-1}{2}+\sqrt{\frac{(n-1)^{2}}{4}+\lambda}} + u_{-}x^{\frac{n-1}{2}-\sqrt{\frac{(n-1)^{2}}{4}+\lambda}},
\end{equation*}
where $u|_{\pm}$ are smooth up to the boundary of $X$.  When the
difference between the two exponents is an integer, then $u_{-}$ is
instead only in $C^{\infty}(X) +
x^{2\sqrt{\frac{(n-1)^{2}}{4}+\lambda}}(\log x)C^{\infty}(X)$.
Moreover, one may specify $u_{\pm}|_{Y_{-}}$, uniquely determining
$u_{\pm}|_{Y_{+}}$.  He showed (cf.\ Theorem 7.21 of \cite{vasy:2007})
that when there are no integer coincidences the map sending the data
\begin{equation*}
  (u_{+}|_{Y_{-}}, u_{-}|_{Y_{-}}) \mapsto (u_{+}|_{Y_{+}}, u_{-}|_{Y_{+}})
\end{equation*}
is a Fourier integral operator associated to the canonical relation
given by bicharacteristic flow from $Y_{-}$ to $Y_{+}$. This is via an
explicit parametrix construction on an appropriate blow-up of $X\times
Y$, and the static model of de Sitter space appears in the
consideration of the rescaled normal operator at the front face of
this blow-up.

Observe now that for $(\xt_{0},\yt_{0})\in X$ away from the boundary
of $X$, the restriction of the kernel of the fundamental solution to
the slice $\xt = \xt_{0}$ agrees with a multiple of the solution of
the Cauchy problem 
\begin{align*}
  &P(\lambda) u = 0, \\
  &u(\xt_{0}, \cdot ) = 0, \\
  &Vu(\xt_{0}, \cdot ) = \delta (y-\yt),
\end{align*}
at least in the region to the future of $(\xt_{0},\yt_{0})$. (Here $V$
is as in equation~\eqref{eq:Vasy-Cauchy-prob}, i.e., a linear
combination of $\xt\pd[\xt]$ and $\xt\pd[\yt]$ transverse to $\{ \xt =
\xt_{0}\}$.) In particular, we can understand the fundamental solution
in this region by understanding the solution operator for the Cauchy
problem, which Vasy studied.

The solution operator for the Cauchy problem is the composition of the
Poisson operator and the Cauchy-to-scattering operator. Vasy showed
that these are Fourier integral operators, and their canonical
relations intersect transversely, so we can compose them. The result
is an operator with canonical relation given by the restriction of the
conormal bundle of the flowout of the light cone, restricted to $\xt =
\xt _{0}$ but lifted to a blown-up space. This blown-up space agrees
with a slice of the space we define later. Vasy's construction blows
up $[X\times Y_{+} , \text{diag}Y_{+}]$, which turns into what we call
$\lcfp$ under this composition. Applying this operator to a delta
distribution gives a conormal distribution on our space, and the log
terms in Vasy's construction become the log terms in our construction,
though this requires careful bookkeeping. Viewing this as a
distribution parametrized by $\xt$ gives the fundamental solution in a
neighborhood of the interior of what we call $\lfp$ away from the
other boundary hypersurfaces.

\section{$0$-geometry}
\label{sec:0-geometry}

Recall from \cite{mazzeo-melrose:1987} the Lie algebra of $0$-vector
fields, 
\begin{equation*}
  \mathcal{V}_{0} = \{ \text{vector fields vanishing at } \pd X\}.
\end{equation*}
The universal enveloping algebra of this Lie algebra is the algebra of
$0$-differential operators, $\Diff[*]_{0}(X)$.

We may use the Lie algebra $\mathcal{V}_{0}(X)$ to define the natural
tensor bundle $\Tzero X$, whose smooth sections are elements of
$\mathcal{V}_{0}(X)$. Its dual, $\Tzero ^{*}X$, has smooth sections
spanned (over $C^{\infty}(X)$) by $\frac{\dx}{x}$ and $\frac{\dy
  _{j}}{x}$.

Recall that $T^{*}X$ is endowed with a canonical $1$-form given by
\begin{equation*}
  \alpha = \xi \dx + \mu \cdot \dy = x\xi \frac{\dx}{x} + x\mu \cdot \frac{\dy}{x}.
\end{equation*}
In particular, $\Tzero ^{*}X$ is endowed with the canonical $1$-form
\begin{equation}
  \label{eq:1-form-single-space-Tzero}
  \azero = \tau \frac{\dx}{x} + \eta \cdot \frac{\dy}{x}.
\end{equation}

Just as $T^{*}X$ is endowed with a symplectic form given by $\omega =
d\alpha$, $\Tzero ^{*}X$ is endowed with the symplectic form $\wzero =
d\azero$.

As in \cite{mazzeo:1988} or \cite{mazzeo-melrose:1987}, we define the
$0$-double space $X\times _{0}X = \dblzero$ as the blown-up space
$[X\times X, \pd \diag]$. $X\times _{0}X$ is a manifold with corners
agreeing with $X\times X$ on the interior. It has three boundary
hypersurfaces -- the left face $\leftface = \{ x=0\}$, which is the
lift of the boundary hypersurface $\{ 0\} \times X$ in $X\times X$;
the right face $\rightface = \{ \xt = 0\}$, which is the lift of
$X\times \{ 0\}$; and the front face $\frontface$, which is the
boundary hypersurface introduced by the blow up construction. Recall
also that $\dblzero$ is equipped with a blow-down map $\beta :
\dblzero \to X\times X$.  A neighborhood of the front face
$\frontface$ is depicted in Figure~\ref{fig:dblspace}.

\begin{figure}[ht]
  \centering
  \begin{tikzpicture}
    \draw [<->] (-4,3) to (0,3) node[anchor=north](yold){$y,\tilde{y}$};
    \draw [->] (-2,3) to (-2, 4.5) node[anchor=west](xtold){$\tilde{x}$};
    \draw [->] (-2,3) to (-1,2) node[anchor=west](xold) {$x$};
    \draw (-2,3) to (0,4) node[anchor=west](diagold){$\diag$};

    \draw [<-, color=StanfordRed, very thick] (0,4.5) .. controls (1, 5) .. (2,4.5);

    \draw [<-] (1,3) to (2,3);
    \draw [->] (4,3) to (5,3) node [anchor=north] (y){$y,\tilde{y}$};
    \draw (2,3) .. controls (3,3.5) .. (4,3) coordinate [pos=0.5,anchor=center] (top);
    \draw (4,3) .. controls (3,2.7) .. (2,3) coordinate [pos=0.5] (bottom);
    \draw [->] (top) to (3, 4.5) node [anchor=west](xt){$\tilde{x}$};
    \draw [->] (bottom) to (4 , 2) node [anchor=west](x){$x$};
    \draw [draw=white,double=black, very thick] (3,3) to (5,4) node [anchor=west](diag){$\diag _{0}$};
  \end{tikzpicture}
  \caption{Passing from $X\times X$ to the $0$-double space $X^2_0 = [X^{2}, \pd\diag ]$.}
  \label{fig:dblspace}
\end{figure}

This construction is perhaps best thought of as an invariant way of
introducing spherical coordinates near the boundary of the diagonal in
$X\times X$. Indeed, a valid coordinate system near the front face
$\frontface$ is given by spherical coordinates: 
\begin{align*}
  \rho _{\frontface} = \sqrt{ x^{2} + \xt ^{2} + |y-\yt|^{2}} \\
  \theta _{0} = \frac{x}{\rho _{\frontface}}, \quad 
  \theta _{n} = \frac{\xt}{\rho_{\frontface}}, \quad 
  \theta ' = \frac{y-y'}{\rho_{\frontface}}, \notag 
\end{align*} 
where $\theta \in S^{n}_{++}$, a quarter sphere. It is often more
convenient to work with projective coordinates near the front face:
\begin{align*}
  s = \frac{x}{\xt}, z = \frac{y-\yt}{\xt}, \xt, \yt .
\end{align*}

In the projective coordinates $(s,z,\xt,\yt)$, we may compute the
lifts of the vector fields $x\pd[x],x\pd[y],\xt\pd[\xt]$, and
$\xt\pd[\yt]$. Indeed, we have 
\begin{align*}
  x\pd[x] \liftsto s\pd[s] &\quad x\pd[y] \liftsto s\pd[z] \\
  \xt\pd[\xt] \liftsto \xt\pd[\xt] - s\pd[s] -z\cdot\pd[z] &\quad \xt\pd[\yt] \liftsto \xt\pd[\yt] -\pd[z]. \notag
\end{align*}

Recall that the front face is the total space of a fibration. Indeed,
it is the total space of a bundle over $Y$ with quarter-sphere fibers.
In local projective coordinates $(s,z,\xt,\yt)$ near the front face,
the fibers are given by $\{ \xt = 0, \yt = \text{ const}\}$. An
asymptotically de Sitter metric on $X$ induces a Lorentzian metric
$-\frac{\ds ^{2}}{s^{2}} + \frac{h(0, \yt, \dz)}{s^{2}}$ on the
interior of the fibers. By an affine change of coordinates in the
fiber, this metric may be written as $-\frac{\ds^{2}}{s^{2}} +
\frac{\dz^{2}}{s^{2}}$, which is the Wick rotation of the (negative
definite) hyperbolic metric on an upper half-space.

Due to our assumptions (A1) and (A2), the space $X$ has two boundary
components $Y_{+}$ and $Y_{-}$, so it is useful at this point to give
names to the components of the left face, right face, and front face.
We call $\ffp$ the component of the front face that comes from blowing
up the diagonal near $Y_{+} \times Y_{+}$, and $\ffm$ the other
component. Similarly, we give the name $\lfp$ to the lift of the
hypersurface $Y_{+} \times X$ and $\lfm$ to the lift of $Y_{-}\times
X$ (and $\rfp$, $\rfm$ denote the lifts of $X\times Y_{+}$ and
$X\times Y_{-}$, respectively).  This naming scheme is illustrated in
Figures~\ref{fig:2-d-dblspace} and \ref{fig:3d-dblspace}.

\begin{figure}[ht]
  \centering
  \begin{tikzpicture}
    \draw [-] (-2,-1) to (-2,2) node [pos=0.5,anchor=east] (rfm) {$\rfm$};
    \draw [-] (2,-2) to (2,1) node [pos=0.5, anchor=west] (rfp) {$\rfp$};
    \draw [-] (2,-2) to (-1,-2) node [pos=0.5, anchor=north] (lfm) {$\lfm$};
    \draw [-] (-2,2) to (1,2) node [pos=0.5, anchor=south] (lfp) {$\lfp$};
    \draw (-1,-2) arc (0:90:1);
    \draw (1,2) arc (180:270:1);
    \draw [-] (1.3,1.3) node [anchor=south west] (ffp) {$\ffp$} to (-1.3,-1.3) node [anchor = north east] (ffm) {$\ffm$} node [pos=0.5, anchor = south east] (diag) {$\diag_{0}$};  
  \end{tikzpicture}
  \caption{A two-dimensional view of $\dblzero$.}
  \label{fig:2-d-dblspace}
\end{figure}

\begin{figure}[ht]
  \centering
  \begin{tikzpicture}
    \coordinate (fl) at (-4,1);
    \coordinate (fr) at (4,1);
    
    \coordinate (ll) at (-4,0);
    \coordinate (lr) at (-1,0);
    \path (ll) to (lr) coordinate [pos = 0.5] (lm) 
    coordinate [pos = 0.33] (llm) 
    coordinate [pos = 0.77] (lrm);

    \coordinate (rl) at (1,0);
    \coordinate (rr) at (4,0);
    \path (rr) to (rl) coordinate [pos = 0.5] (rm)
    coordinate [pos = 0.33] (rrm)
    coordinate [pos = 0.77] (rlm);

    \coordinate (tl) at (-1.5,2);
    \coordinate (tr) at (1.5,2);
    \path (tl) to (tr) coordinate [pos = 0.5] (tm);

    \coordinate (bl) at (-1.5,-2);
    \coordinate (br) at (1.5,-2);
    \path (bl) to (br) coordinate [pos = 0.5] (bm);
  
    \draw [-] (ll) to (llm);
    \draw [-] (lr) to (lrm);

    \draw [-] (rl) to (rlm);
    \draw [-] (rr) to (rrm);

    \draw [-] (tl) to (tr);
    \draw [-] (bl) to (br);
    
    \draw [-] (ll) to (tl) node [pos = 0.7, anchor = east] {$\rfm$};
    \draw [-] (lm) to (tm);
    \draw [-] (lr) to (tr);
    
    \draw [-] (ll) to (bl) node [pos = 0.7, anchor = east] {$\lfm$};
    \draw [-] (lm) to (bm);
    \draw [-] (lr) to (br);

    \draw [-] (intersection of lr--tr and tl--rl) to (rl);
    \draw [-] (intersection of lr--tr and tm--rm) to (rm);
    \draw [dashed] (intersection of lr--tr and tm--rm) to (tm);
    \draw [dashed] (intersection of lr--tr and tl--rl) to (tl);
    \draw [-] (tr) to (rr) node [pos = 0.3, anchor = west] {$\lfp$};

    \draw [-] (intersection of lr--br and bl--rl) to (rl);
    \draw [-] (intersection of lr--br and bm--rm) to (rm);
    \draw [dashed] (intersection of lr--br and bm--rm) to (bm);
    \draw [dashed] (intersection of lr--br and bl--rl) to (bl);
    \draw [-] (br) to (rr) node [pos = 0.3, anchor = north west] {$\rfp$};

    \filldraw [fill = white, draw = black] (llm)  .. controls +(0.8,0.4) .. (lrm) .. controls +(-0.4,-0.3) .. (llm) coordinate [pos = 0.5] (lmb);
    
    \filldraw [fill = white, draw = black] (rrm) .. controls +(-0.8,0.4) .. (rlm)
    ..controls +(0.4, -0.3) .. (rrm);

    \draw (lm) to (fl) node [anchor = south] {$\ffm$};
    \draw (rm) to (fr) node [anchor = south] {$\ffp$};
  \end{tikzpicture}
  \caption{Another view of $\dblzero$. The horizontal lines represent the $y,\tilde{y}$ axes.}
  \label{fig:3d-dblspace}
\end{figure}
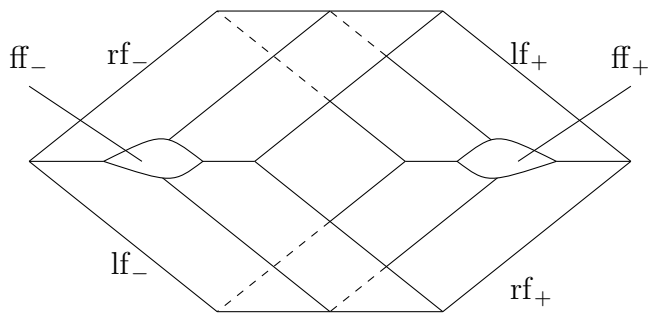

We should also recall the trivial half-density bundles $\hd (X)$ and
$\hd (\dblzero)$, which are trivial bundles. In local coordinates
$(x,y)$, $\hd (X)$ is trivialized by the global section $\gamma =
\left|\dg\right|^{1/2}$. The bundle $\hd (X\times X)$ is trivialized
by $\upsilon = \left|\dg_{L}\dg_{R}\right|^{1/2}$, where $\dg_{L}$ and
$\dg_{R}$ are the lifts of the densities from the left and right
factors of $X\times X$. Up to a nonvanishing factor, the Jacobian
determinant of the blow-down map $\dblzero \to X\times X$ is $r^{n}$,
so $\upsilon$ lifts to $r^{n/2}\mu$, where $\mu$ is a nonvanishing
section of the standard half-density bundle $\hd (\dblzero)$. For
example, we may take $\mu = \left| \dr\dtheta\dyt\right|^{1/2}$ in
local polar coordinates. In projective coordinates, we may take $\mu =
\left| \ds\dz\dxt\dyt\right|^{1/2}$, and then $\upsilon$ lifts to
$\xt^{n/2}\mu$ (times a nonvanishing factor). In this paper, we adopt
the convention that $\upsilon$ and $\gamma$ are flat, i.e., for a
differential operator $L$, $L(u\gamma) = (Lu)\gamma$ and $L(u\upsilon)
= (Lu)\upsilon$. In particular, 
\begin{equation*}
  L(u\mu) = L(\xt^{-n/2}u\upsilon) = (L(\xt ^{-n/2}u))\upsilon = (\xt^{n/2}L(\xt^{-n/2}u))\mu .
\end{equation*}
Our operator $P(\lambda)$ commutes with $\xt$, so $P(\lambda)(u\mu) =
(P(\lambda)u)\mu$.

On the fibers of the front face, elements of $\Diff[*]_{0}(X)$ lift to
differential operators on $\dblzero$ that restrict to differential
operators on the fibers of the front face by setting $\xt = 0$. We
call this restriction to the fiber over $p\in Y$ the \emph{normal
  operator} and denote it as $N_{p}(A)$ for $A\in \Diff[*]_{0}(X)$. As
an example, in projective coordinates the normal operator of
$P(\lambda)$ is given by
\begin{equation}
  \label{eq:normal-op-of-P}
  N_{p}(P(\lambda)) = (s\pd[s])^{2} - (n-1)(s\pd[s]) + s^{2}\lap _{z} - \lambda .
\end{equation}

The inclusions $\frontface \subset \dblzero$ and $\leftface \subset
\dblzero$ induce inclusions on the tangent bundles $T\frontface
\subset T_{\frontface}\dblzero$ and $T\leftface\subset
T_{\leftface}\dblzero$. These inclusions induce natural projections on
the cotangent bundle: 
\begin{equation*}
  \pi_{\frontface}:T_{\frontface}^{*}\dblzero \to T^{*}\frontface \quad \text{and}\quad \pi_{\leftface}T_{\leftface}^{*}\dblzero \to T^{*}\leftface. 
\end{equation*}

Let $2\dblzero$ denote the doubling of $\dblzero$ across $\frontface$,
and let $T^{*}_{\frontface}\dblzero$ denote the restriction of the
cotangent space to the front face. As in \cite{joshi-sabarreto:2001},
we say that a smooth closed conic Lagrangian submanifold
$\Lambda\subset T^{*}\dblzero$ is \emph{extendible} if it intersects
$T^{*}_{\frontface}\dblzero$ transversely. Recall that in this case
there is a smooth closed conic Lagrangian submanifold
$\Lambda^{e}\subset T^{*}2\dblzero$ such that 
\begin{equation*}
  \Lambda = \Lambda^{e}\cap T^{*}\dblzero, \quad \Lambda^{0} = \pi _{\frontface}(\Lambda \cap T^{*}_{\frontface}\dblzero).
\end{equation*}
Note that there are many choices for the extension $\Lambda^{e}$.

We recall a result \cite{joshi-sabarreto:2001}:
\begin{lem}[Lemma 2.1 of \cite{joshi-sabarreto:2001}]
  \label{lem:0G-extendible-Lagrangian-restriction}
  Let $\Lambda\subset T^{*}\dblzero$ be an extendible Lagrangian. Then
  $\Lambda_{0} = \pi_{\frontface}(\Lambda\cap
  T^{*}_{\frontface}\dblzero)$ is a Lagrangian submanifold of
  $T^{*}\frontface$. 
\end{lem}

For future use, we collect a bit more information about the symplectic
structure of $T^{*}\dblzero$. The space $\Tzero ^{*}X \times \Tzero
^{*}X$ is endowed with the symplectic form 
\begin{equation}
  \label{eq:0G-symplectic-form-on-X2}
  \omega = \pi _{1} ^{*} \omega _{X}+ \pi _{2}^{*}\omega _{X},
\end{equation}
where $\omega _{X}$ is the symplectic form on $\Tzero ^{*}X$ (coming from the canonical $1$-form (\ref{eq:1-form-single-space-Tzero})) and $\pi_{j}$ the projection of the $j$th copy of $\Tzero ^{*}X$.  Moreover, the blow-down map $\beta$ induces a smooth map 
\begin{equation*}
  \Tzero^{*}X \times \Tzero^{*}X \to T^{*}\dblzero 
\end{equation*}
which is an isomorphism on the interior.

The identification of $T^{*}X$ and $\Tzero ^{*}X$ over the interior of
$X$ then induces a smooth map $\Tzero ^{*}X \times \Tzero ^{*}X \to
T^{*}\dblzero$ over the interior of $X\times X$. This map identifies
the $1$-forms 
\begin{equation*}
  \tau \frac{\dx}{x} + \tilde{\tau}\frac{\dxt}{\xt} + \eta \cdot \frac{\dy}{x} + \tilde{\eta}\cdot \frac{\dyt}{\xt} \quad\text{ and }\quad \sigma \ds + \xi \dxt + \zeta \cdot \dz + \mu \cdot \dyt,
\end{equation*}
and so we must have
\begin{align}
  \label{eq:0G-cotangent-coordinates}
  \sigma = \frac{\tau\xt}{x} , &\quad \xi = \frac{\tau}{\xt} + \frac{\tilde{\tau}}{\xt} + \frac{\eta}{x}\cdot \frac{y-\yt}{\xt} \\
  \zeta = \frac{\xt \eta}{x}, &\quad \mu = \frac{\eta}{x} + \frac{\tilde{\eta}}{\xt}. \notag 
\end{align} 
Observe that in these coordinates, the symplectic form $\omega$ takes
a familiar form: 
\begin{align*}
  \differential{\sigma}\wedge \ds + \dxi \wedge \dxt + \differential{\zeta}\wedge\dz + \dmu \wedge \dyt.
\end{align*} 

\section{Polyhomogeneity and conormal distributions}
\label{sec:polyhomogeneity}

In order to consider the pushforward of a distribution, we need the
notion of a b-fibration. We recall from \cite{melrose:1992} that an
\emph{interior b-map} $f:X\to Y$ of manifolds with corners is a function
mapping $X\to Y$ and $X\setminus \pd X \to Y\setminus \pd Y$ such that
each boundary defining function for $Y$ pulls back to a sum of
products of boundary defining functions for $X$. More precisely,
suppose that $M_{1}(X)$ is the set of boundary hypersurfaces of $X$,
and, for each $G\in M_{1}(X)$, let $I(G)$ be the ideal of functions
vanishing at $G$. Suppose that $H$ is a boundary hypersurface for $Y$
and $I(H)$ is the ideal of functions vanishing on $H$. We say that $f:
X\to Y$ is an interior b-map if 
\begin{equation*}
  f^{*}I(H) = \prod _{G\in M_{1}(X)}I(G)^{e_{f}(G,H)}.
\end{equation*}
Here $e_{f}(G,H)$ is a collection of nonnegative integers that we call
the exponent matrix of $f$. We set $\null (e_{f}) = \{ H \in
M_{1}(X) : e_{f}(H,G) = 0 \forall G\in M_{1}(Y)\}$.

For an interior b-map $f$, the differential $f_{*}: T_{p}X \to
T_{f(p)}Y$ extends by continuity to the b-differential
\begin{equation*}
  f_{*}: {}^{b}T_{p}X \to {}^{b}T_{f(p)}Y.
\end{equation*}

Recall that a \emph{b-fibration} between two manifolds with corners is one
that is both b-normal and a \emph{b-submersion}. A b-submersion is a b-map
with surjective b-differential. A b-normal map is one such that the
b-differential is surjective as a map ${}^{b}N_{x}H \to
{}^{b}N_{f(x)}G$ on the interior of each boundary hypersurface $H$.
Here $G$ is the face that contains the image of $H$.

We say that a discrete subset $E \subset \complexes\times \naturals$
is an \emph{index set} if $(z_{j}, k_{j}) \in E$ with $|(z_{j} , k_{j}) |\to
\infty$ implies that $\Re z_{j} \to \infty$, $(z,k)\in E$ implies that
$(z+p, k)\in E$ for all $p\in \naturals$, and if $(z,k)\in E$, then
$(z,p)\in E$ for all $p\in \naturals$ with $0\leq p < k$. We say that
$\E = \{ E_{H} : H \text{ is a boundary hypersurface of }M\}$ is an
index family if each $E_{H}$ is an index set.

Our construction below shows that certain distributions have
asymptotic expansions. We make this notion more precise by recalling
the definition of a polyhomogeneous distribution. 
\begin{defn}
  \label{defn:phg-dist} 
  A \emph{polyhomogeneous distribution} is a distribution that is smooth on the interior
  of $M$, and, near each boundary hypersurface $H$ of $M$ has an
  asymptotic expansion of the form
  \begin{equation*}
    \sum _{(s,l)\in E_{H}}x_{H}^{s}(\log x_{H})^{l} a_{sl}(x,y),
  \end{equation*}
  where $x_{H}$ is a defining function for $H$, $E_{H}$ is an index set for $H$, and $a_{sl}$ are smooth functions independent of $x_{H}$.  Near the corners of $M$, we require that polyhomogeneous distributions have an appropriate product-type expansion.
\end{defn}
For an extended discussion of polyhomogeneous distributions, we refer
the reader to \cite{mazzeo:1991}, \cite{melrose:1992}, or \cite{melrose:1993}. We adopt the convention that $a_{sl}$ should take
values in the half-density bundles.

Recall that if $f$ and $g$ are polyhomogeneous with index families
$\mathcal{F}$ and $\mathcal{G}$, then $fg$ is polyhomogeneous with
index set $\mathcal{F}+\mathcal{G}$.

We also require the following lemma, which can be found in
\cite{melrose:1992} or \cite{melrose:1993}, and allows us to
understand the polyhomogeneity of the pushforward of a distribution:
\begin{lem}
  \label{lem:phg-b-fibration-pushforward}
  If $f: X\to Y$ is a b-fibration between compact manifolds with
  corners, then for any index family $\mathcal{K}$ for $X$ with
  \begin{equation*}
    \Re \mathcal{K}(H) > -1 \quad \text{if }H\in \null (e_{f}),
  \end{equation*}
  the pushforward gives
  \begin{align*}
    f_{*}&: \phg{\mathcal{K}}{X}{\Omega \otimes f^{*}E}\to \phg{\mathcal{J}}{Y}{\Omega \otimes E} , \\
    \mathcal{J} &= f_{\#}\mathcal{K}. \notag
  \end{align*}
\end{lem}

Here $\mathcal{J}= f_{\#}\mathcal{K}$ is the pushforward of the index
family $\mathcal{K}$. In particular, for a boundary hypersurface $H$
of $Y$, 
\begin{equation*}
  \mathcal{J}(H) = \extcup _{G\in M_{1}(X), e_{f}(G,H)\neq 0}\left\{ (z/e_{f}(G,H), p): (z,p)\in \mathcal{K}(G)\right\},
\end{equation*}
where $\extcup$ is the extended union defined by $(z,p)\in K\extcup J$
if and only if $(z,p)\in K\cup J$ or $p = p' + p'' +1$ with $(z,p')\in
K$ and $(z,p'')\in J$.

We recall the definition of the space of conormal distributions (see,
for example, \cite{hormander:vol3} or \cite{melrose:1992}). If $M$ is
a manifold with corners and $Z$ is an interior $p$-submanifold, then
$I^{p}(M;Z)$ is the space of conormal distributions of order $p$
associated to $Z$. Here we only need the case when $Z$ has codimension
one and meets the boundary hypersurfaces transversely. In local
coordinates $(x,y)$ where $M$ is given by $x_{i}\geq 0$ for $i =
1,\ldots ,k$ and $Z$ is given by $y_{1} = 0$, then $u\in I^{p}(M;Z)$
has the form
\begin{equation*}
  \int_{\reals} e^{i y_{1}\eta} a(x,y,\eta) \deta ,
\end{equation*}
with $a$ a half-density valued classical symbol of order $p+n/4$ and
$n = \dim M$ (note that we are using the order convention consistent
with the orders of Lagrangian distributions in
\cite{melrose-uhlmann:1979}). Because $x$ behaves as a parameter, the
boundaries cause no problems here.

We also need several refinements of this notion. The first is when we
allow the symbol $a$ to have an asymptotic expansion (in $x$) near the
boundary of $M$. We need the space $\phgalt{\E}\CDf{p}{M}{Z}$ of
conormal distributions associated to $Z$ with polyhomogeneous
expansion at the boundary of $M$. These are merely conormal
distributions in the sense above, but where we allow the symbol $a$ to
have a polyhomogeneous expansion in the $x$ variables.

The second generalization we require is to allow the symbol to have a
mild type of singularity at $\eta = 0$. In particular, we allow
symbols of the form $\sum_{j} (\eta + i0)^{r-j}a_{j}$. The singularity
at $0$ only affects the growth of the distribution as $y\to
\pm\infty$, and so we may avoid complications by multiplying our
distribution by a cutoff function supported near $y=0$. We need this
generalization later in order to prove a lemma about the support of
our distributions.

\section{The {H}amilton vector fields}
\label{sec:hamilt-vect-fields}

By analogy with the definition of the Hamilton vector field, we recall
from \cite{joshi-sabarreto:2001} the definition of the \emph{$0$-Hamilton
vector field}. 
\begin{defn}
  \label{def:HVF-0-HVF}
  Given $p\in C^{\infty}(T^{*}X)$, the $0$-Hamilton vector field of
  $p$, denoted $\Hzero _{p}$, is defined by
  \begin{equation*}
    \wzero (\cdot, \Hzero _{p}) = dp .
  \end{equation*}
\end{defn}

In local coordinates where $\azero$ is given by
(\ref{eq:1-form-single-space-Tzero}), $\Hzero _{p}$ is given by
\begin{equation*}
  \Hzero _{p} = x\frac{\pd p}{\pd \tau} \frac{\pd}{\pd x} + x\frac{\pd
    p}{\pd \eta}\cdot \frac{\pd}{\pd y} - \left( x\frac{\pd p}{\pd x}
    + \frac{\pd p}{\pd \eta}\cdot \eta \right) \frac{\pd}{\pd \tau} -
  \left( x\frac{\pd p}{\pd y} - \frac{\pd p}{\pd \tau} \eta\right)
  \cdot \frac{\pd }{\pd \eta}. 
\end{equation*}

We are interested in the operator $P(\lambda)$, whose principal symbol
is the length function \begin{equation*}
  p = -x^{2}\xi ^{2} + x^{2}(h_{0}(y,\mu) + xh_{1}(x,y,\mu)) = -\tau
  ^{2} + h_{0}(y,\eta) + xh_{1}(x,y,\eta), 
\end{equation*}
where we have written $h(x,y,\mu) = h_{0}(y,\mu) + xh_{1}(x,y,\mu)$.
This is also the principal symbol of $P$ when $P$ is treated as an
operator acting on the left factor of the product space $X\times X$.

In particular, in the coordinates given by
(\ref{eq:0G-cotangent-coordinates}), $p$ pulls back to
\begin{equation*}
  \tilde{p} = -s^{2}\sigma ^{2} + s^{2}h_{0}(\yt + \xt z, \zeta) +
  s^{3}\xt h_{1}(\xt s, \yt + \xt z, \zeta). 
\end{equation*}

Thus if $\wdbl$ is the lift of the symplectic form
(\ref{eq:0G-symplectic-form-on-X2}) (as in
\cite{joshi-sabarreto:2001}) to $T^{*}\dblzero$, $\Hzero _{p}$ lifts
to $H_{\tilde{p}}$, which is given by $\wdbl (\cdot,
H_{\tilde{p}})=d\tilde{p}$. In these coordinates, we may write
\begin{align*}
  H_{\tilde{p}} = &\frac{\pd \tilde{p}}{\pd \sigma} \frac{\pd}{\pd s}
  - \frac{\pd \tilde{p}}{\pd s}\frac{\pd}{\pd \sigma} + \frac{\pd
    \tilde{p}}{\pd \zeta}\cdot \frac{\pd}{\pd z} - \frac{\pd
    \tilde{p}}{\pd z}\cdot \frac{\pd}{\pd \zeta} \\ 
  &\quad +\frac{\pd \tilde{p}}{\pd \xi}\frac{\pd}{\pd\xt} - \frac{\pd
    \tilde{p}}{\pd \xt}\frac{\pd}{\pd\xi} + \frac{\pd\tilde{p}}{\pd
    \mu}\cdot \frac{\pd}{\pd y} - \frac{\pd\tilde{p}}{\pd y}\cdot
  \frac{\pd}{\pd\mu} \notag  
\end{align*}
An elementary computation then yields
\begin{align*}
  H_{\tilde{p}}= &-2s^{2}\sigma \frac{\pd}{\pd s} - s\left(-2\sigma
    ^{2}  + 2 h_{0} + 3s\xt h_{1} + s^{2}\xt ^{2}\frac{\pd h_{1}}{\pd
      x} \right) \frac{\pd}{\pd \sigma}  \\ 
  &\quad+ s^{2}\left( \frac{\pd h_{0}}{\pd \eta} + s\xt \frac{\pd
      h_{1}}{\pd \eta}\right) \cdot \frac{\pd}{\pd z} - s^{2}\left(\xt
    \frac{\pd h_{0}}{\pd y}  + s\xt ^{2} \frac{\pd h_{1}}{\pd y}
  \right) \cdot \frac{\pd}{\pd \zeta} \\ 
  &\quad - s^{2}\left(  z \cdot \frac{\pd h_{0} }{\pd y} + sh_{1} +
    s^{2}\xt \frac{\pd h_{1}}{\pd x} + s\xt z \cdot \frac{\pd
      h_{1}}{\pd y}\right) \frac{\pd}{\pd \xi} - s^{2}\left( \frac{\pd
      h_{0}}{\pd y} + s\xt \frac{pd h_{1}}{\pd y}\right) \cdot
  \frac{\pd}{\pd \mu}, 
\end{align*}
where $h_{0}$ and $h_{1}$ are evaluated at $(\xt s, \yt + \xt z,
\zeta)$.

Note that the characteristic set of $P$ on $\dblzero$ is the set where
$\tilde{p} = 0$, i.e., the set where $s^{2}\sigma ^{2} = s^{2}(h(\yt +
\xt z, \zeta) + s\xt h_{1}(s\xt , \yt + \xt z, \zeta))$. In this set,
$H_{\tilde{p}}$ has the form 
\begin{align}
  \label{eq:HVF-tilde-p-in-char-set}
  H_{\tilde{p}} = &s^{2}\bigg( -2\sigma \frac{\pd}{\pd s} - \left(\xt
    h_{1} + s\xt ^{2}\frac{\pd h_{1}}{\pd x} \right) \frac{\pd}{\pd
    \sigma}  + \left( \frac{\pd h_{0}}{\pd \eta} + s\xt \frac{\pd
      h_{1}}{\pd \eta}\right) \cdot \frac{\pd}{\pd z} \notag\\ 
  &\quad - \left(\xt \frac{\pd h_{0}}{\pd y}  + s\xt ^{2} \frac{\pd
      h_{1}}{\pd y} \right) \cdot \frac{\pd}{\pd \zeta} - \left(
    \frac{\pd h_{0}}{\pd y} + s\xt \frac{pd h_{1}}{\pd y}\right) \cdot
  \frac{\pd}{\pd \mu}\notag\\ 
  &\quad - \left(  z \cdot \frac{\pd h_{0} }{\pd y} + sh_{1} +
    s^{2}\xt \frac{\pd h_{1}}{\pd x} + s\xt z \cdot \frac{\pd
      h_{1}}{\pd y}\right) \frac{\pd}{\pd \xi} \bigg), 
\end{align}

In particular, we have proved the following proposition: 
\begin{prop}
  \label{prop:HVF-factors-in-char-set}
  Inside the characteristic set of $P$ on $\dblzero$, $H_{L}=
  H_{\tilde{p}} = s^{2}\tilde{H}_{L}$, where $\tilde{H}_{L}$ is a
  smooth vector field tangent to the front face that is nondegenerate
  at $s=0$. Similarly, in coordinates $(x, y, \tilde s = s^{-1},
  \tilde{z} = s^{-1}z)$, the Hamilton vector field for the lift from
  the right factor can be written $H_{R} = \tilde{s}^{2}
  \tilde{H}_{R}$, where $\tilde{H}_{R}$ is nondegenerate at $\tilde{s}
  = 0$. 
\end{prop}

The fact that we may factor a power of $s^{2}$ from the Hamilton
vector field is useful in the next section.

\section{The {L}agrangians}
\label{sec:lagrangians}

The two Lagrangians that interest us are the lift of the diagonal and
the flowout from the characteristic set of $P$ within this conormal
bundle by the Hamilton vector field $H_{L}$.

The coordinates on phase space given by
(\ref{eq:0G-cotangent-coordinates}) imply that the conormal bundle of
the diagonal, 
\begin{equation*}
  \{ x=\xt, y=\yt, \tau = -\tilde{\tau} , \eta = -\tilde{\eta}\},
\end{equation*}
 lifts to
\begin{equation*}
  \Lambda_{0} = \left\{ s = 1, z = 0, \xi = 0, \mu = 0\right\}.
\end{equation*}
The Lagrangian submanifold $\Lambda_{0}$ intersects
$T^{*}_{\frontface}\dblzero$ transversely at
\begin{equation*}
  \Lambda_{0}^{0} = \left\{ s=1,z=0,\xi=0,\mu=0,\xt =0\right\} =
  T^{*}_{\diag_{0} \cap \frontface}\frontface , 
\end{equation*}
and so $\Lambda_{0}$ is extendible.

We now set $\Lambda_{1}$ to be the forward flowout of $H_{\tilde{p}}$
from $\Lambda_{0} \cap \charset (P)$. Because we may write
$H_{\tilde{p}} = s^{2}\tilde{H}_{\tilde{p}}$, with
$\tilde{H}_{\tilde{p}}$ a smooth nondegenerate vector field tangent to
$\ffp$, $\Lambda_{1}$ is a smooth submanifold of $T^{*}\dblzero$,
intersecting $T^{*}_{\ffp}\dblzero$ transversely. $\Lambda_{1}$ is
Lagrangian by general theory (see, for example,
\cite{dimassi-sjostrand:1999}). Thus $\Lambda_{1}$ is an extendible
Lagrangian near $\ffp$.

To see that $\Lambda_{1}$ is extendible near $\ffm$, we require the
following proposition: 
\begin{prop}
  \label{prop:Lag-invariance}
  The Lagrangian $\Lambda_{0}$ is invariant under the flow of $H_{L} -
  H_{R}$. Moreover, $\Lambda_{1}$ is also the flowout by $H_{R}$ of
  $\Lambda_{0} \cap \charset (P)$. 
\end{prop}

\begin{proof}
  A straightforward calculation shows that $H_{L} - H_{R}$ preserves
  $\Lambda_{0}$. Indeed, though this is a nonzero vector field, the
  coefficients of $\pd[s]$, $\pd[z]$, $\pd[\xi]$, and $\pd[\mu]$ all
  vanish at $\Lambda_{0}$. Note further that $\charset (P_{L}) \cap
  \Lambda_{0} = \charset (P_{R})\cap \Lambda_{0}$, and so $H_{L} -
  H_{R}$ preserves this set.

  In order to see that $\Lambda_{1}$ is also the flowout of $H_{R}$,
  we observe that $H_{L}$ and $H_{R}$ must commute because they are
  lifts (from $X\times X$) of commuting vector fields. In particular,
  because $H_{R}-H_{L}$ is tangent to $\Lambda_{0}\cap \charset
  (P_{L})$, it is also tangent to $\Lambda_{1}$. In particular,
  $H_{R}$ is also tangent to $\Lambda_{1}$. $H_{R}$ is not tangent to
  $\Lambda_{0}$, so its flowout must also be $\Lambda_{1}$.
\end{proof}

An argument similar to the one above then shows that $\Lambda_{1}$ is
extendible near $\ffm$ (with $H_{L}$ replaced by $H_{R}$). (For what
we use below, the argument used above suffices, as we only need that
$\Lambda_{1}$ is extendible in a neighborhood of the diagonal.)

The projection of $H_{L}$ in $\charset (P)$ to $T^{*}\ffp$ is given by
\begin{align*}
  H_{L} = &s^{2}\left( -2\sigma \frac{\pd}{\pd s} + 2\frac{\pd h}{\pd
      \zeta}\cdot\frac{\pd}{\pd z} - \frac{\pd h}{\pd y} \cdot
    \frac{\pd}{\pd\mu} \right)  , 
\end{align*}
so that $\Lambda_{1}^{0}$ is given by the flowout of this vector field
from $\Lambda_{0}^{0}$ in $T^{*}\ffp$.

Note that the assumption (A3) implies that the projection of
$\Lambda_{1}$ to the base $\dblzero$ is a smooth submanifold. We call
this submanifold $\lightcone = \pi \Lambda_{1}$. This is a smooth
codimension one submanifold. Note further that $\lightcone$ is the
boundary of an open subset of $\dblzero$, which we call $\lightcone
^{int}$.

\section{Paired {L}agrangian distributions}
\label{sec:pair-lagr-distr}

We are interested in a calculus of paired Lagrangian distributions
adapted to the $0$-geometry. These are distributions with a model form
\begin{align}
  \label{eq:PL-model-form}
  u &= \int _{0}^{\infty} \int _{\reals^{n}} e^{i (x - \xt -
    t\xt)\frac{\xi}{\xt} + i (y - \yt)\cdot
    \frac{\mu}{\xt}}a\left(t,\frac{x}{\xt},
    \frac{y-\yt}{\xt},\xt,\yt,\xi,\mu\right)\dt\dxi\dmu \\ 
  &= \int _{0}^{\infty}\int _{\reals^{n}} e^{i(s - 1 - t)\xi + i
    Y\cdot \mu}a(t,s,Y,\xt,\yt,\xi,\mu)\dt\dxi\dmu , \notag 
\end{align}
where $a$ is a symbol.

We briefly recall the definition of a paired Lagrangian distribution
from \cite{melrose-uhlmann:1979}. Suppose $M$ is a manifold and
$L_{0},L_{1}$ are conic Lagrangian submanifolds of $T^{*}M$
intersecting cleanly in codimension $1$. We say that $u\in
\PLflong{m}{M}{L_{0}}{L_{1}}{\hd (M)}$ if it is a distributional
$\frac{1}{2}$-density modeled microlocally on 
\begin{equation*}
  \int _{0}^{\infty}\int _{\reals^{n}}e^{i(x_{1}-s)\xi_{1} + i x' \cdot \xi '}a(s,x,\xi)\dxi \ds ,
\end{equation*}
where $a$ is a symbol of order $m + \frac{1}{2}-\frac{n}{4}$. In the
general case, we require the phase function to parametrize the
Lagrangian pair $(L_{0},L_{1})$ and $a$ to be a symbol of order $m +
\frac{1}{2} + \frac{n-2N}{4}$, where $N$ is the number of variables
required to parametrize $L_{1}$. In this model, the phase function
parametrizes $L_{0}= N^{*}\{ 0\}$, while $L_{1} = \{ x_{1}\geq 0, x' =
0, \xi _{1}= 0\}$.

We now recall what it means to parametrize a cleanly intersecting
Lagrangian pair. Suppose that $L_{0},L_{1}\subset T^{*}M$ are two
closed conic Lagrangian submanifolds intersecting cleanly in
codimension $1$ and that $\pd L_{1} = L_{1}\cap L_{0}$. Suppose $\phi:
M_{x} \times \reals^{N}_{\xi}\setminus \{ 0\}\to \reals$ is
homogeneous of degree $1$ in $\xi$, and let 
\begin{equation*}
  C = \{ (x,\xi)~:~x\in M, \xi \in \reals^{N}\setminus\{0\}, d_{\xi}\phi (x,\xi) = 0\}.
\end{equation*}
We say $\phi$ is nondegenerate if $d\frac{\pd\phi}{\pd \xi_{j}}$,
$j=1,\ldots,N$, are linearly independent at any point in $C$. We say
that $\phi$ parametrizes a single conic Lagrangian submanifold $L$ if
$L$ is locally the image of the map $C\to T^{*}M$ given by $(x,\xi)
\mapsto (x,d_{x}\phi (x,\xi))$. A simple example is the case when $L =
N^{*}\{ 0\}$, in which case we may take $\phi (x,\xi) = x\cdot \xi$.

In the case of an intersecting Lagrangian pair, we say that $\phi$ is
a nondegenerate parametrization of $(L_{0},L_{1})$ near $q\in \pd
L_{1}$ if 
\begin{equation*}
  \phi (x,s,\xi) = \phi_{0}(x,\xi) + s\phi _{1}(x,s,\xi), \quad \xi\in \reals^{n},
\end{equation*}
where $q = (x_{0}, d_{x}\phi (x,0,\xi))$, $d_{(s,\xi)}\phi (x,0,\xi) =
0$, $\phi_{0}$ is a nondegenerate parametrization of $L_{0}$ near $q$,
and $\phi$ parametrizes $L_{1}$ near $q$ for $s>0$ with $ds$,
$d\frac{\pd\phi}{\pd \xi_{j}}$, and $d\frac{\pd\phi}{\pd s}$ all
linearly independent at $q$. In particular, the phase function
$(1-s-t)\xi + Y\cdot \mu$ parametrizes $(\Lambda_{0},\Lambda_{1})$
near the front face of $\dblzero$.

\begin{defn}
  \label{defn:PL-basic-distributions}
  We say that $u\in \PLlong$ is a \emph{paired Lagrangian distribution} of
  order $m$ associated to $\Lambda_{0}$ and $\Lambda_{1}$ on
  $\dblzero$ if it may be written as the restriction to $\dblzero$ of
  some distribution in
  $\PLflong{m}{2\dblzero}{\Lambda_{0}^{e}}{\Lambda_{1}^{e}}{\hd
    (2\dblzero)}$. 
\end{defn}

\begin{defn}
  \label{defn:PL-support-away-from-sf}
  The class $\PLzerolong$ consists of the paired Lagrangian
  distributions that are supported away from the side faces of
  $\dblzero$.

  We similarly define the classes
  $\LDflong[0]{m}{\dblzero}{\Lambda_{0}}{\hd (\dblzero)}$ and
  $\LDflong[0]{m}{\dblzero}{\Lambda_{1}}{\hd (\dblzero)}$ as
  restrictions of distributions on $2\dblzero$ that are supported away
  from the side faces. 
\end{defn}

We typically shorten $\PLlong$ and
$\LDflong[0]{m}{\dblzero}{\Lambda}{\hd (\dblzero)}$ by suppressing the
half-density factor.

The following proposition follows immediately from Proposition 4.1 of
\cite{melrose-uhlmann:1979}. 
\begin{prop}
  \label{prop:PL-LDs-are-PLDs}
  If $u\in \PLzero$, then $\WF (u) \subset \Lambda_{0}\cup
  \Lambda_{1}$. If $B$ is a properly supported $b$-pseudodifferential
  operator of order $0$ on $\dblzero$, then
  \begin{align*}
    WF'(B)\cap \Lambda_{0} = \emptyset \implies Bu \in
    \LDf[0]{m}{\dblzero}{\Lambda_{1}\setminus \pd \Lambda_{1}}{\hd
      (\dblzero)} , \ \text{and}\\ 
    WF'(B)\cap \Lambda_{1} = \emptyset \implies Bu \in
    \LDf[0]{m-1/2}{\dblzero}{\Lambda_{0}\setminus \pd\Lambda_{1}}{\hd
      (\dblzero)}. 
  \end{align*}
\end{prop}

Note that the orders on $\Lambda_{0}$ and $\Lambda_{1}$ differ by
$\frac{1}{2}$. This can be seen by integrating by parts once in $s$.

Just as in \cite{melrose-uhlmann:1979}, this gives us two symbol maps.
For $u\in \PLzero$, 
\begin{align*}
  \sigma ^{(1)}(u) &\in C^{\infty}_{m+n/4}(\Lambda_{1}\setminus
  \pd\Lambda_{1}; \hd\otimes L_{1}) \\ 
  \sigma ^{(0)}(u) &\in C^{\infty}_{m-1/2 + n/4}(\Lambda_{0}\setminus
  \pd\Lambda_{1}; \hd \otimes L_{0}), 
\end{align*}
where the subscript indicates the degree of homogeneity and $L_{i}$ is
the Maslov bundle over $\Lambda_{i}$. Admissible symbols are subject
to a compatibility condition at $\pd \Lambda_{1}$ as in
\cite{melrose-uhlmann:1979}, but we do not need the explicit form of
this condition here. We do, however, use the fact that if the principal symbol
of $u$ vanishes then $u$ is one order better.

For these classes to be well-defined, we must show that these classes
are independent of the choice of extensions $\Lambda_{0}^{e}$ and
$\Lambda_{1}^{e}$. (It was already shown in
\cite{joshi-sabarreto:2001} that the classes
$\LDf[0]{m}{\dblzero}{\Lambda}{\hd (\dblzero)}$ are independent of the
choice of extension.)

\begin{lem}
  \label{lem:PL-classes-independent-of-extension}
  The class $\PLzero$ is independent of the choice of extensions
  $\Lambda_{0}^{e}$ and $\Lambda_{1}^{e}$. 
\end{lem}

\begin{proof}
  Suppose that $\Lambda_{0}^{e}$ and $\tilde{\Lambda}_{0}^{e}$ both
  extend $\Lambda_{0}$, while $\Lambda_{1}^{e}$ and
  $\tilde{\Lambda}_{1}^{e}$ extend $\Lambda_{1}$. Let $u \in
  \PLf[0]{m}{2\dblzero}{\tilde{\Lambda}_{0}^{e}}{\tilde{\Lambda}_{1}^{e}}{\hd
    (2\dblzero)}$. We may find $v\in
  \PLf[0]{m}{2\dblzero}{\Lambda_{0}^{e}}{\Lambda_{1}^{e}}{\hd
    (2\dblzero)}$ with the same symbol on $\Lambda_{0}\cup
  \Lambda_{1}$.

  Because $u$ and $v$ have the same symbol there, $u-v$ is order
  $m-1$ on $\Lambda_{0}\cup \Lambda_{1}$. We now iteratively solve
  away principal symbols to find $w$ with $u-w$ of order $-\infty$ on
  $\Lambda_{0}\cup \Lambda_{1}$, i.e., such that $u-w$ is a smooth
  function on $\dblzero$ up to the front face, but supported away from
  the side faces. $u-w$ is then clearly the restriction of an element
  of
  $\PLf[0]{-\infty}{2\dblzero}{\Lambda_{0}^{e}}{\Lambda_{1}^{e}}{\hd
    (2\dblzero)}$, proving the claim.
\end{proof}

The following proposition follows easily from Proposition 5.4 of
\cite{melrose-uhlmann:1979}: 
\begin{prop}
  \label{prop:PL-behavior-under-characteristic-ops}
  Suppose $A$ is a properly supported b-differential operator of order
  $m$ on $\dblzero$ and is such that $\sigma (A)$ vanishes on
  $\Lambda_{1}$. If $u\in \xt^{p}\PLzero[k]$, then $Au =
  \xt^{p}h+\xt^{p}g$, with $h\in \LDf[0]{m + k -
    1/2}{\dblzero}{\Lambda_{0}}{\hd (\dblzero)}$ and $g\in
  \PLzero[k+m-1]$. 
\end{prop}

\begin{proof}
  We start by noting that the action of $\Diff[m]_{b}(\dblzero)$
  (extended to $2\dblzero$) commutes with restriction to $\dblzero$.
  This follows from the observation that $x\pd[x] \chi_{\{x > 0\}} =
  0$, where $\chi_{\{x >0\}}$ is the characteristic function of a
  half-plane.
  
  We start by writing $u = \xt^{p}\tilde{u}|_{\dblzero}$, where
  $\tilde{u}\in
  \PLf[0]{k}{2\dblzero}{\Lambda_{0}^{e}}{\Lambda_{1}^{e}}{\hd
    (2\dblzero)}$. We apply Proposition 5.4 of
  \cite{melrose-uhlmann:1979} to $\tilde{u}$, giving us $A\tilde{u} =
  \tilde{h}+\tilde{g}$.

  In particular, we have that $Au = \xt^{p}h + \xt^{p}g +
  [A,\xt^{p}]u$. The operator $A$ is a b-differential operator, so we know that
  $[A,\xt^{p}] = \xt ^{p}B$, where $B$ is a b-differential operator of
  order $m-1$. We thus have that $Au = \xt^{p}h+\xt^{p}(g + Bu)$,
  where $h$ and $g$ are the restrictions of $\tilde{h}$ and
  $\tilde{g}$ to $\dblzero$. $g,Bu \in\PLzero[k+m-1]$, proving the
  claim. 
\end{proof}

We also need another statement from Proposition 5.4 of
\cite{melrose-uhlmann:1979}, namely that the principal symbol of the
distribution $g$ above is given by the action of the Hamilton vector
field of the principal symbol of $A$ on the principal symbol of $u$.
We only require this on the interior of the double space and so we do
not prove it here.

\begin{defn}
  \label{def:PL-normal-op}
  We define the \emph{normal operator} of a paired Lagrangian distribution
  (or a Lagrangian distribution, as in \cite{joshi-sabarreto:2001}) to
  be the restriction of its Schwartz kernel to the front face
  $\frontface$. 
\end{defn}

Note that this restriction is well-defined by wavefront considerations and
the transversality of $\Lambda_{0}$ and $\Lambda_{1}$ to $\frontface$.
Moreover, by considering the model form (\ref{eq:PL-model-form}), we
find that the restriction is a paired Lagrangian distribution in the
class $\PLf{m}{\frontface}{\Lambda_{0}^{0}}{\Lambda_{1}^{0}}{\hd
  (\dblzero)}$.

Just as in \cite{mazzeo-melrose:1987}, we have a short exact sequence.
\begin{prop}
  \label{prop:PL-normal-op-SES}
  The normal operator induces exact sequences
  \begin{equation}
    \label{eq:LD-normal-op-SES}
    0 \to \xt \LDf[0]{m}{\dblzero}{\Lambda_{0}}{\hd(\dblzero)} \to
    \LDf[0]{m}{\dblzero}{\Lambda_{0}}{\hd (\dblzero)} \to
    \LDf[0]{m}{\frontface}{\Lambda_{0}^{0}}{\hd (\dblzero)} \to 0 
  \end{equation}
  and 
  \begin{equation}
    \label{eq:PL-normal-op-SES}
    0 \to \xt\PLzero[m] \to \PLzero[m] \to
    \PLf[0]{m}{\frontface}{\Lambda_{0}^{0}}{\Lambda_{1}^{0}}{\hd
      (\dblzero)}\to 0 
  \end{equation}
  such that for any differential operator $P\in \Diff[m]_{0}(X)$ and
  any paired Lagrangian distribution $u$,
  \begin{equation}
    \label{eq:PL-normal-op-distributes}
    N_{p}(Pu) = N_{p}(P)N_{p}(u).
  \end{equation}
\end{prop}

\begin{proof}
  The exactness of the sequences (\ref{eq:LD-normal-op-SES}) and
  (\ref{eq:PL-normal-op-SES}) follows from Taylor expansion in $\xt$
  near $\frontface$.

  The proof of (\ref{eq:PL-normal-op-distributes}) is identical to the
  one in \cite{mazzeo-melrose:1987}. 
\end{proof}

Also note that because $\lightcone$ is an embedded submanifold,
elements of $\LDf[0]{k-1/2}{\dblzero}{\Lambda_{1}}{\hd(\dblzero)}$ may
be identified with distributions of order $k-1/2$ conormal to
$\lightcone$.

\section{Another blow-up}
\label{sec:another-blow-up}

In order to understand the solutions of the transport equations, we
must introduce another blow-up. Because our solutions eventually have
differing asymptotic behaviors along the light cone and on the
interior of the light cone, we blow up the boundary of the light cone
(i.e., at the left and right faces). This blow-up always makes sense
locally near the front face, but to make sense of it globally we use
assumption (A3).

Because the boundary of the light cone meets the corner $\lfp \cap
\rfm$, the order in which we blow up the two submanifolds $\lightcone
\cap \lfp$ and $\lightcone \cap \rfm$ matters. We deal with this
situation by blowing up the submanifold $\lightcone \cap \lfp \cap
\rfm$ of the corner. We include this discussion primarily for
completeness, as this piece of the construction is unnecessary as long
as we restrict to forward-directed data. Indeed, the product of the
pullback of a forward-directed function and any tempered distribution
on $\dblspace$ will vanish to all orders at this new face.

We first define what we refer to as the intermediate double space.
This is the space on which we solve the transport equations for the
conormal singularity. If we restrict to data supported away from
$Y_{-}$, this space suffices for our entire construction.

\begin{defn}
  \label{defn:ABU-new-double-space}
  We define the intermediate double space $\dblspacet = \left[
    \dblzero , \lightcone \cap \rightface , \lightcone \cap \leftface
  \right]$. 
\end{defn}

This is a new manifold with corners. We will call $\lcfp$ the lift of
$\lightcone \cap \leftface$ and $\lcfm$ the lift of $\lightcone \cap
\rightface$.  

Though we may think of this new manifold as an invariant way of
introducing polar coordinates near $\lightcone\cap \leftface$,
projective coordinates are more convenient for our applications. Near
$\lcfp$ and $\ffp$ but away from $\lfp$, we may use coordinates $\rho
/ s$, $s$, and $\xt$. Similarly, near $\lcfm$ and $\ffm$ and away from
$\rfm$, we may use $\rho / \tilde{s}$, $\tilde{s}$, and $x$ (and the
remaining $\tilde{z}$ and $y$ variables), where $\tilde{s} = \xt / x$.

Away from the front face and the corner $\lfp\cap\rfm$, this is just
the blow-up of an intersection of two hypersurfaces $\{ \rho = 0\}$
and $\{x=0\}$ (or $\{ \xt = 0\}$). Near $\lcfm$ but away from $\lcfp$
and $\ffp$, $\rho / \xt$ and $\xt$ are valid coordinates, while near
$\lcfp$ away from $\lcfm$ and $\ffm$, $\rho / x$ and $x$ are valid
coordinates. Near their intersection but away from $\lfp$ and $\rfm$,
$\rho / (\xt x)$, $\xt$, and $x$ are valid.

Because $\lcfp$ and $\lcfm$ intersect, the order in which we performed
the blow-up matters. Near the interiors of the new faces the two
spaces are locally diffeomorphic, but they are not globally
diffeomorphic. This is relevant when understanding the behavior of our
parametrix near intersections of these faces, so we instead work on
what we call the \emph{full double space}. In this space we first blow up
$\lightcone\cap \lfp\cap\rfm$ and then perform the other two blow-ups.
Performing this first blow-up separates the lifts of $\lightcone\cap
\lfp$ and $\lightcone\cap \rfm$, so the order in which we perform the
other two blow-ups is now irrelevant.

\begin{defn}
  \label{defin:ABU-full-double-space}
  The full double space is $\dblspace = \left[ \dblzero, \lightcone
    \cap \lfp\cap \rfm, \lightcone \cap \lfp , \lightcone \cap
    \rfm\right]$. 
\end{defn}

This is another new manifold with corners and is the space on which
the final parametrix lives. We call the additional corner $\xf$, for
scattering face. We give it this name because this face is related to
Vasy's construction of the scattering operator in \cite{vasy:2007}.
Near the interior of $\xf \cap \lcfp$, we may use coordinates given by
$\rho / (x\xt)$, $x/xt$, and $\xt$, while near the interior of $\xf
\cap \lcfm$, we may use $\rho / (x\xt)$, $\xt / x$, and $x$. We again
emphasize that this blow-up is unnecessary if we restrict to data
supported away from $Y_{-}$.  Figure~\ref{fig:newdblspace} depicts a neighborhood of
$\ffp$ in $\dblspace$ (or, indeed, in $\dblspacet$).

\begin{figure}[ht]
  \centering
  \begin{tikzpicture}
    \draw [<-] (1,3) to (2,3);
    \draw [->] (4,3) to (5,3) node [anchor=west] (y){$y,\tilde{y}$};
    \draw (2,3) .. controls (3,3.5) .. (4,3) coordinate [pos=0.5,anchor=center] (top)
    coordinate [pos=0.25, anchor=center](leftlc)
    coordinate [pos=0.65, anchor=center](rightlc)
    coordinate [pos=0.2, anchor=center](llbuc)
    coordinate [pos=0.3, anchor=center](lrbuc)
    coordinate [pos=0.6, anchor=center](rlbuc)
    coordinate [pos=0.7, anchor=center](rrbuc);
    \draw (4,3) .. controls (3,2.7) .. (2,3) coordinate [pos=0.5] (bottom);
    \draw [->] (top) to (3, 4.5) node [anchor=west](xt){$\tilde{x}$};
    \draw [->] (bottom) to (4 , 2) node [anchor=west](x){$x$};
    \draw [draw=white,double=black,very thick] (3,3) to (5,4) node [anchor=west](diag){$\Delta _{0}$};
    \draw (3,3) to (leftlc);
    \draw (3,3) to (rightlc);
    \filldraw[draw = black,fill=white]{(llbuc) .. controls +(0.15,-0.05) .. (lrbuc) .. controls +(-0.25,0.65) .. ([shift={(-0.25, 1)}]lrbuc) .. controls +(-0.075,-0.1) .. ([shift={(-0.25, 1)}]llbuc) .. controls +(0,-0.35) .. (llbuc)};
    \filldraw[draw=black,fill=white]{(rrbuc) .. controls +(-0.15,-0.05) .. (rlbuc) .. controls +(0.25, 0.65) .. ([shift={(0.25, 1)}]rlbuc) .. controls +(0.075, -0.1) .. ([shift={(0.25,1)}]rrbuc) .. controls +(0,-0.35) .. (rrbuc)};
  \end{tikzpicture}
  \caption{The double space $\dblspace$ near $\ffp$.}
  \label{fig:newdblspace}
\end{figure}

Up to a smooth nonvanishing factor, the Jacobian determinant of the
blow-down map $\dblspacet \to \dblzero$ is given by
$r_{\lcfp}r_{\lcfm}$, so sections of the bundle $\hd(\dblzero)$ lift
to sections of $(r_{\lcfp}r_{\lcfm})^{1/2}\hd(\dblspacet)$. Similarly,
sections of $\hd (\dblzero)$ lift to sections of
$r_{\xf}r_{\lcfp}^{1/2}r_{\lcfm}^{1/2}\hd(\dblspace)$.

One can see that the fundamental solution does not lift nicely to
$\dblzero$ by considering the forward fundamental solution for the
Klein-Gordon equation on the half-space $\reals^{n}_{+} =
(0,\infty)_{s}\times \reals^{n-1}_{z}$ equipped with the Lorentzian
metric $\frac{-\ds^{2}+\dz^{2}}{s^{2}}$. Finding the forward
fundamental solution requires solving the equation 
\begin{align*}
  P(\lambda)u &= \delta (s-1)\delta(z), \\
  u &\equiv 0 \quad \text{ for }s>1 .
\end{align*}
This solution can be constructed explicitly by taking the Fourier
transform in the $z$ variables, which transforms the equation into a
Bessel ordinary differential equation for each value of the dual
variable to $z$. By applying a stationary phase argument to the
inverse Fourier transform of this family of Bessel functions, one can
see that the asymptotic behavior of the forward fundamental solution
is qualitatively different from its behavior on the interior of the
light cone, justifying the blow-up.

\section{Outline of the construction}
\label{sec:outline-construction}

We present now an outline of the parametrix construction, which
consists of five steps. At each step we construct an approximation of
the fundamental solution that captures the ``worst'' remaining
singularity and yields an error term that is less singular.

We simultaneously construct a distribution that is both a left and a
right parametrix for our operator. In other words, we construct a
distribution $K$ such that $P(\lambda)K = I + R_{1}$ and
$P(\lambda)^{t}_{R}K = I + R_{2}$. Here the subscript R denotes the
operator acting on the right variables and $P(\lambda)^{t}$ is the
transpose of $P(\lambda)$ as in equation~(\ref{eq:op-transpose}). The
distributions $R_{1}$ and $R_{2}$ are smoothing operators that vanish
to infinite order at relevant boundary hypersurfaces and should be
considered ``negligible''.

First we show that the construction in \cite{melrose-uhlmann:1979} is
valid in a neighborhood of the diagonal, smooth down to the front face
of $\dblzero$. The Schwartz kernel of the identity is a distribution
conormal to the lift of the diagonal to $\dblzero$, and so the
fundamental solution of $P(\lambda)$ is a paired Lagrangian
distribution associated to the lift of the diagonal and its flowout
via the Hamilton vector field, uniformly down to the front face. This
is the most singular piece of the distribution but is the most
standard portion of the construction. After cutting off this
distribution outside a neighborhood of the lifted diagonal, this step
yields a remainder term that is Lagrangian on the interior, uniformly
down to the front face, and supported away from the side faces. Our
assumption (A3) (which is unnecessary if we are only considering the
construction in a neighborhood of $Y_{+}$) guarantees that the
Lagrangian remainder term is in fact a conormal distribution
associated to the projection of the flowout of the light cone.

In the second step of the construction, we solve away the conormal
error term to obtain a remainder that is smooth on the interior of the
double space. Because the light cone is characteristic for
$P(\lambda)$, this reduces to solving a sequence of transport ordinary
differential equations. These may be solved on the interior, but
solving each successive transport equation on $\dblzero$ causes a loss
of an order of decay at the boundary. In other words, with each
successive improvement of regularity at the interior, the behavior of
the solution at the boundary becomes worse.

We instead solve the transport equations on the intermediate double space
$\dblspacet$. We solve these equations on $\dblspacet$ rather than the
full double space $\dblspace$ because transport equations have a
slightly simpler form here. Indeed, the main modification needed to
handle de Sitter space is to solve the equations on $\dblspace$ rather
than $\dblspacet$. This step then yields a remainder term that is
smooth on the interior and polyhomogeneous at the light cone faces
$\lcfp$ and $\lcfm$ and the scattering face $\xf$.

The third step solves away part of the polyhomogeneous remainder at
the light cone faces. We write down a formal series expansion for the
solution and solve term by term, which yields a sequence of ordinary
differential equations. The numbers $s_{\pm} (\lambda)$ are the
indicial roots of these equations near the side faces of $\dblspace$,
so $x^{s_{\pm}(\lambda)}$ solve the equations to first order. In other
words, the fundamental solution has an expansion at the side faces
starting with $x^{s_{\pm}(\lambda)}$. The remainder term from this
step vanishes to all orders at the light cone face and has an
expansion at the side faces.

In the fourth step we solve away the remaining error term at the front
face and side faces. This reduces to another formal power series
calculation, and was carried out in \cite{vasy:2007}. The distribution
we obtain by adding the distributions found in the four steps then
solves the left (right) equation up to an error term that is smooth
and vanishing to all orders at the ``plus'' (``minus'') boundary
faces.

In the fifth and final step, we remove this last error term via an argument showing that
the fundamental solution itself must be in our class of distributions.

\section{The diagonal singularity}
\label{sec:diagonal-singularity}

The aim of this section is to solve away the diagonal singularity of
the fundamental solution, leaving us with a Lagrangian error. The
method here is similar to the one found in
\cite{joshi-sabarreto:2001}. We use the normal operator to solve away
the paired Lagrangian singularity at the front face up to $O(\xt
^{\infty})$, which allows us to invoke the construction in
\cite{melrose-uhlmann:1979}.

\begin{prop}
  \label{prop:PL-construction}
  Given $f\in \xt^{p}\LDf[0]{k}{\dblzero}{\Lambda_{0}}{\hd
    (\dblzero)}$, there is a paired Lagrangian distribution $u\in
  \xt^{p}\PLzero[-3/2 + k]$ such that
  \begin{equation*}
    P(\lambda )u - f\in \xt^{p}\LDf[0]{-1/2 +k}{\dblzero}{\Lambda_{1}}{\hd(\dblzero)}.
  \end{equation*}
  In particular, if $f$ is the Schwartz kernel of the identity
  operator, then
  \begin{equation*}
    f = \delta (s-1)\delta(z) \xt ^{-n/2}\mu \in \xt
    ^{-n/2}\LDf[0]{0}{\dblzero}{\Lambda_{0}}{\hd (\dblzero)}, 
  \end{equation*}
  and there is a paired Lagrangian distribution $u\in \xt
  ^{-n/2}\PLzero[-3/2]$ such that $P(\lambda)u - f \in
  \xt^{-n/2}\LDf[0]{-1/2}{\dblzero}{\Lambda_{1}}{\hd(\dblzero)}$.
\end{prop}

\begin{proof}
  On the interior of the manifold we may use the construction in
  \cite{melrose-uhlmann:1979}, so we localize near the front face.
  Because we are only considering distributions supported away from
  the side faces for now, we do not need separate arguments for
  dealing with $\ffp$ and $\ffm$.

  We start by fixing a cutoff function $\chi \in C^{\infty}(\dblzero)$
  such that $\chi \equiv 1$ in a neighborhood of the diagonal, but
  supported away from the side faces of $\dblzero$. We also note now
  that $P(\lambda)$ commutes with multiplication by $\xt$.

  Consider now the normal operator $N_{p}(P(\lambda ))$ acting on the
  fiber $\frontface_{p}$ over the front face. In the projective
  coordinates $(s , z, \xt, \yt)$, this is given by
  equation~\eqref{eq:normal-op-of-P}.
 
  It is easy to verify that $N_{p}(P(\lambda ))$ satisfies the
  assumptions required to apply Proposition 6.6 of
  \cite{melrose-uhlmann:1979}. We may thus find $\kappa\in
  \PLf{k-3/2}{\frontface}{\Lambda_{0}^{0}}{\Lambda_{1}^{0}}{\hd
    (\dblzero)}$ such that $N_{p}(P(\lambda ))\kappa(p) -
  N_{p}(\xt^{-p}f)\in C^{\infty}(\frontface)$. Multiplying $\kappa$ by
  the cutoff function $\chi$ gives \[
  N_{p}(P(\lambda ))(\chi\kappa(p)) - N_{p}(x^{-r}f) \in \LDf[0]{k-1/2}{\frontface}{\Lambda_{1}^{0}}{\hd(\dblzero)}.
  \]

  By appealing to the short exact sequence
  (\ref{eq:PL-normal-op-SES}), we may find some paired Lagrangian
  distribution $u_{0}\in \xt^{p}\PLzero[k-3/2]$ such that
  $N(\xt^{-p}u_{0}) = \xt^{-p}\chi\kappa$. $P(\lambda )$ is
  characteristic on $\Lambda_{1}$, so we may appeal to Propositions
  \ref{prop:PL-behavior-under-characteristic-ops} and
  \ref{prop:PL-normal-op-SES} to write
  \begin{equation*}
    P(\lambda )u_{0} - f = v_{0} + w_{0},
  \end{equation*}
  where $v_{0}\in \xt^{p+1}
  \LDf[0]{k}{\dblzero}{\Lambda_{0}}{\hd(\dblzero)}$ and $w_{0}\in
  \xt^{p}\LDf[0]{k-1/2}{\dblzero}{\Lambda_{1}}{\hd(\dblzero)}$.

  The distribution $v_{0}$ is supported away from the side face, so
  $\xt^{-1}v_{0}\in
  \xt^{p}\LDf[0]{k}{\dblzero}{\Lambda_{0}}{\hd(\dblzero)}$. We may now
  again use equation (\ref{eq:LD-normal-op-SES}), Proposition 6.6 of
  \cite{melrose-uhlmann:1979}, and equation
  (\ref{eq:PL-normal-op-SES}) to find $u_{1}\in \xt^{p}\PLzero[k-3/2]$
  such that
  \begin{equation*}
    N(P(\lambda ))N(\xt^{-p}u_{1}) = N(\chi\xt^{-1}v_{0}).
  \end{equation*}

  In particular, we may then again appeal to Propositions
  \ref{prop:PL-behavior-under-characteristic-ops} and
  \ref{prop:PL-normal-op-SES} to write
  \begin{equation*}
    P(\lambda )(\chi(u_{0} + \xt u_{1})) = v_{1} + w_{0} + \xt w_{1},
  \end{equation*}
  with $v_{1}\in
  \xt^{p+2}\LDf[0]{k}{\dblzero}{\Lambda_{0}}{\hd(\dblzero)}$ and
  $w_{i}\in
  \xt^{p}\LDf[0]{k-1/2}{\dblzero}{\Lambda_{1}}{\hd(\dblzero)}$.

  Iterating this process gives us distributions $u_{j}$ such that
  \begin{equation*}
    P(\lambda )\left(\chi\sum _{j=0}^{N}\xt ^{j}u_{j}\right) = v_{N} + \sum _{j=0}^{N}\xt^{j}w_{j}
  \end{equation*}
  with $v_{N}\in
  \xt^{p+N+1}\LDf[0]{k}{\dblzero}{\Lambda_{0}}{\hd(\dblzero)}$ and
  $w_{j}\in
  \xt^{p}\LDf[0]{k-1/2}{\dblzero}{\Lambda_{1}}{\hd(\dblzero)}$.
  
  We may now asymptotically sum the $\xt ^{j}u_{j}$ to find $u\in
  \xt^{p}\PLzero[k-3/2]$ such that
  \begin{equation*}
    P(\lambda )(\chi u) - f = v + w, \quad v\in
    \xt^{\infty}\LDf[0]{k}{\dblzero}{\Lambda_{0}}{\hd(\dblzero)}, w\in
    \xt^{p}\LDf[0]{k-1/2}{\dblzero}{\Lambda_{1}}{\hd(\dblzero)}. 
  \end{equation*}
  We now extend $v$ by zero to $2\dblzero$, and use proposition 6.6 of
  \cite{melrose-uhlmann:1979} to find
  \begin{equation*}
    \tilde{u}\in \PLzero[k-3/2]
  \end{equation*}
  such that $P(\lambda )\tilde{u} + v \in C^{\infty}(2\dblzero)$. In
  particular, $P(\lambda )(\chi \tilde{u}) - v = \tilde{w}\in
  \xt^{p}\LDf[0]{k-1/2}{\dblzero}{\Lambda_{1}}{\hd(\dblzero)}$. The
  distribution $\chi u + \chi \tilde{u}$ then satisfies
  \begin{equation}
    \label{eq:PL-diag-result}
    P(\lambda )(\chi u + \chi \tilde{u}) - f \in
    \xt^{p}\LDf[0]{k-1/2}{\dblzero}{\Lambda_{1}}{\hd(\dblzero)}. 
  \end{equation}
\end{proof}

We should note here that solving the transport equations for the
symbols on $\Lambda_{1}$ (i.e., when invoking Proposition 6.6 of
\cite{melrose-uhlmann:1979}) actually fixes the behavior of the
solution in both components of $\Lambda_{1} \setminus 0$. In
particular, we may arrange it so that the symbol on $\Lambda_{1}$ has
an expansion in decreasing powers of $(\eta + i0)$ (here $\eta$ is the
fiber variable in $\Lambda_{1} = N^{*}\lightcone$). This allows us to
guarantee that our parametrix $u$ and the error term in
(\ref{eq:PL-diag-result}) are supported on the interior of the light
cone. This is not surprising because the exact solution must also be
supported on the interior of the light cone due to the finite speed of
propagation for the wave equation.

The following lemma is also useful:
\begin{lem}
  When $f = \kappa _{I}$ is the Schwartz kernel of the identity
  operator, then the distribution $u$ constructed in this section is
  also a left parametrix for $P$, i.e., $P_{R}^{t}u - f = \xt
  ^{-n/2}r$, where $r$ is smooth in a neighborhood of the diagonal.
  Here $P_{R}^{t}$ denotes the transpose operator for $P$ acting on
  the right factor. \end{lem}

\begin{proof}
  This is an application of the symbol calculus for paired Lagrangian
  distributions.

  We start by observing that $(P_{L} - P_{R}^{t} )\kappa_{I} = 0$
  because the identity operator commutes with any operator. In other
  words, we must have that
  \begin{equation*}
    \int (P_{L}\kappa_{I})v = Pv = \int \kappa_{I} (Pv) = \int (P_{R}^{t}\kappa_{I})v.
  \end{equation*}
  
  Now let $v = (P_{L} - P_{R}^{t})u \in \xt ^{-n/2}\PLzero[k]$. We
  know that $P_{L}v$ is smooth in a neighborhood of the diagonal, down
  to the front face, because $P_{L}$ and $P_{R}^{t}$ commute. We now
  invoke the symbol calculus:
  \begin{equation*}
    0 = \sigma (P_{L} v)|_{\Lambda_{0}\setminus \pd \Lambda_{1}} =
    \sigma(P_{L})\sigma(v)|_{\Lambda_{0}\setminus \pd\Lambda_{1}}, 
  \end{equation*}
  and $\sigma (P_{L}) \neq 0$ on $\Lambda_{0}\setminus
  \pd\Lambda_{1}$, so $\sigma(v)=0$ on this set. This fixes an initial
  condition for $\sigma (v)$, i.e., $\sigma (v)|_{\pd\Lambda_{1}} = 0$.
  We now use the form of the transport equation
  \begin{equation*}
    0 = \sigma (P_{L}v )|_{\Lambda_{1}} = (i\mathcal{L}_{H_{L}} +
    c)\sigma (v)|_{\Lambda_{1}}, 
  \end{equation*}
  with initial condition $0$, to conclude that $\sigma
  (v)|_{\Lambda_{1}} = 0$. We may use this argument for any $k$, so we
  must have that $v\in \xt ^{-n/2}\PLzero[-\infty]$, proving the
  claim. 
\end{proof}

\section{The transport equation}
\label{sec:transport-equation}

We now wish to solve away the error from
Proposition~\ref{prop:PL-construction}. We call this error $r$ and
note that $r\in
r_{\frontface}^{p}\LDf[0]{k-1/2}{\dblzero}{\Lambda_{1}}{\hd(\dblzero)}$.
By solving a transport equation for some finite time and multiplying
by a cutoff function, we may assume that this error is supported in a
neighborhood of the side faces. Viewed as a conormal distribution near
$\ffp$, $r$ may be written as
\begin{equation}
  \label{eq:transport-model-form-for-conormal-dist}
  \int e^{i\frac{\rho}{s}\eta}a(s,\frac{\rho}{s},\theta, \xt, \yt , \eta) \deta.
\end{equation}
where $\rho$ is a defining function for the light cone and $a$ is a
classical symbol of order $(k-\frac{1}{2}) + \frac{n-1}{2}$. In fact,
as noted in Section~\ref{sec:diagonal-singularity}, we may assume that
the symbol has an expansion in decreasing powers of $(\eta + i0)$.
Near $\lcfp\cap \lcfm$, we must replace $\frac{\rho}{s}$ with
$\frac{\rho}{x\xt}$.

We first fix a defining function $\rho$ for the submanifold
$\lightcone$. Assumption (A3) guarantees that $\lightcone$ is an
embedded submanifold of $\dblzero$. Because $N^{*}\lightcone$ is
characteristic for $P(\lambda)$ and $d\rho$ spans $N^{*}\lightcone$,
we must have that $\hat{g}(d\rho,d\rho) = 0$ at $\lightcone$, i.e.
\begin{equation}
  \label{eq:TE-expression-for-drho}
  \hat{g}(d\rho, d\rho) = \rho b,
\end{equation}
where $b$ is a smooth function.

In coordinates near $\pd X = Y$, we may write
\begin{align*}
  P(\lambda) = (s\pd[s])^{2} - (n-1)s\pd[s] + \frac{\xt
    s\pd[s]\sqrt{h}}{\sqrt{h}}s\pd[s] +s^{2}\lap _{h} - \lambda. 
\end{align*}

Our ansatz is that $u$ is polyhomogeneous at $\lcfp$ and $\lcfm$ and
conormal to $\lightcone$, and so we seek an expression of the form
(\ref{eq:transport-model-form-for-conormal-dist}) with $a$ a classical
symbol of order $k-\frac{3}{2}+\frac{n-1}{2} = k + \frac{n}{2} - 2$
with an expansion in powers of $(\eta + i0)$, i.e., of the form
\begin{equation}
  \label{eq:TE-ansatz}
  u \asympto \sum _{j\geq 0} \int e^{i\frac{\rho}{s}\eta}(\eta + i0) ^{k + \frac{n}{2} - 2 - j}a_{j}\deta .
\end{equation}
This is because our error $r$ from
Section~\ref{sec:diagonal-singularity} is a Lagrangian distribution of
order $k-\frac{1}{2}$ associated to the conormal bundle of
$\lightcone$. $\lightcone$ is characteristic for $P(\lambda)$, so we
expect the solution of $P(\lambda)u = r$ to be conormal of one order
better, i.e., of order $k-\frac{3}{2}$. Moreover, at each step, we
multiply the symbols $a_{j}$ by a compactly supported smooth function
in $\frac{\rho}{s}$ that is identically $1$ near $0$, which makes the
singularity of $(\eta + i0)^{k+\frac{n}{2}-2-j}$ at $0$ superfluous.
Note that we could equivalently insist on expressing our ansatz in
powers of $\left( \frac{\rho}{s}\right)_{+}$. If $k=0$, then the top
power seen here would be $-1 + 2 - \frac{n}{2} = 1-\frac{n}{2}$, the
same powers seen in the construction of the fundamental solution of
the wave equation on Minkowski space (in our convention, $n$ is the
total dimension of the spacetime).

\begin{lem}
  \label{lem:TE-apply-op-to-ansatz-result}
  Suppose $u$ is of the form (\ref{eq:TE-ansatz}), and $a_{j} =
  \tilde{a}_{j}\nu$, where $\nu$ is a fixed nonvanishing section of
  $\hd(\dblspacet)$. If we write $\gamma_{j} = k + \frac{n}{2}-2 -j$,
  then near $\lcfp$ away from $\lcfp\cap \lcfm$ we have
  \begin{align}
    \label{eq:TE-op-to-ansatz-result}
    P(\lambda)u \asympto &\sum_{j=0}^{\infty} \int
    e^{i\frac{\rho}{s}\eta} (\eta + i0) ^{\gamma_{j} + 1}
    \left(-2i\hat{g}(d\rho, sd\tilde{a}_{j}) -
      i(\pd[s]\rho)(n-3-2\gamma_{j}+O(s))\tilde{a}_{j}
    \right)\deta\,\nu \notag\\ 
    &+\sum_{j=0}^{\infty}\int e^{i\frac{\rho}{s}\eta} (\eta + i0)
    ^{\gamma_{j}} \bigg( P(\lambda) + (\gamma_{j} + O(s))s\pd[s] -
    (\gamma_{j}+1)(n-3-2\gamma_{j}) \notag\\ 
    &\quad \left.  + \frac{1}{2}\left(n-\frac{1}{2}\right) +
      O(s)s\pd[z] + O(s)\right)\tilde{a}_{j}\deta\,\nu 
  \end{align}
  where $O(s)$ is taken to mean an element of $sC^{\infty}$ and
  $\left(s = \frac{x}{\xt}, z= \frac{y-\yt}{\xt}\right)$ are
  projective coordinates near the front face. 
\end{lem}

\begin{proof}
  We first show the result near $\ffp \cap \lfp$.

  Write $u_{j} = \tilde{u}_{j}\nu$, where $\nu$ is a fixed
  trivialization of $\hd (\dblspacet)$. Say, for concreteness, that
  $\nu = r_{\lcfp}^{-1/2}\xt^{-n/2}\upsilon$. We then have that
  \begin{align*}
    P(\lambda) (\tilde{u}_{j}\nu) &=
    \left(P(\lambda)\tilde{u}_{j}\right)\upsilon +
    r_{\lcfp}^{1/2}\xt^{n/2} \left([P(\lambda),
      r_{\lcfp}^{-1/2}\xt^{-n/2}]\right)\tilde{u}_{j}. 
  \end{align*}

  We note also that near $\lcfp$ but away from $\lfp$,
  $r_{\lcfp}=s(1+\alpha s)$, where $\alpha$ is smooth, and so we may
  easily calculate this commutator:
  \begin{align*}
    (r_{\lcfp}^{1/2}\xt ^{n/2})[P(\lambda),
    r_{\lcfp}^{-1/2}\xt^{-n/2}] = \left( -1 + O(s)\right) s\pd[s] +
    O(s)s\pd[z] + \frac{1}{2}\left( n-\frac{1}{2} + O(s)\right). 
  \end{align*}

  We now use this calculation to drop the density factor. We apply
  $P(\lambda)$ to our ansatz and use
  equation~(\ref{eq:TE-expression-for-drho}). Integration by parts
  allows us to exchange powers of $\frac{\rho}{s}$ for decreasing
  powers of $(\eta + i0)$, as $\frac{\rho}{s}e^{i\frac{\rho}{s}\eta} =
  \frac{1}{i}\pd[\eta]e^{i\frac{\rho}{s}\eta}$. If we write $\gamma
  _{j} = k + \frac{n}{2}-2-j$, this yields
  \begin{align*}
    P(\lambda)\tilde{u} &= \left( P(\lambda) + (-1 + O(s))s\pd[s] +
      \frac{1}{2}\left(n- \frac{1}{2} + O(s)\right)\right) \tilde{u}
    \\ 
    &\asympto \sum _{j=0}^{\infty} \int e^{i\frac{\rho}{s}\eta} (\eta
    + i0)^{\gamma_{j} + 2}\hat{g}(d\rho,d\rho)\tilde{a}_{j}\deta \\ 
    &\quad+ \sum_{j=0}^{\infty}\int e^{i\frac{\rho}{s}\eta} (\eta +
    i0)^{\gamma_{j} + 1} \bigg(-2i\hat{g}(d\rho, sd\tilde{a}_{j}) \\ 
    &\quad\quad\quad- i(\pd[s]\rho)(n-1-2\gamma_{j} - 3)\tilde{a}_{j} -
    i(\pd[s]\rho)\tilde{a}_{j} + O(s)\tilde{a}_{j}\bigg)\deta \\ 
    &\quad+ \sum_{j=0}^{\infty} \int e^{i\frac{\rho}{s}\eta} (\eta +
    i0)^{\gamma_{j}}\bigg( P(\lambda) + 2(\gamma_{j}+1)s\pd[s] -
    (\gamma_{j}+1)\left( n-1-2\gamma_{j}-1\right) \\ 
    &\quad\quad\quad + (\gamma_{j}+1) - s\pd[s] + \frac{1}{2}\left(
      n-\frac{1}{2}\right) + O(s)s\pd[s] +O(s)s\pd[z] +
    O(s)\bigg)\tilde{a}_{j}\deta, 
  \end{align*}
  where $O(s)$ is taken to mean an element of $sC^{\infty}$. We now
  use that $\hat{g}(d\rho,d\rho) = \rho b$ and integrate by parts to
  prove the first part of the claim.

  We finish the proof with a similar calculation, where we change our
  operator to $P_{R}^{t}$ and our ansatz to be one of the following
  two forms: 
  \begin{align*}
    u &\asympto \sum _{j=0}^{\infty}\int
    e^{i\frac{\rho}{\tilde{s}}\eta}a_{j}(\eta + i0)^{\gamma_{j}}\deta,
    \\ 
    u &\asympto \sum _{j=0}^{\infty}\int
    e^{i\frac{\rho}{x\xt}\eta}a_{j}(\eta + i0)^{\gamma_{j}}\deta , 
  \end{align*}
  where $\tilde{s}= \frac{\xt}{x}$, and $P_{R}^{t}$ is given by
  equation~(\ref{eq:op-transpose}). We use here that $r_{\lcfm} = \st
  (1 + \alpha \st)$ in the first case and that $r_{\lcfm} = \xt
  (1+\tilde{\alpha} \xt)$ and $r_{\lcfp} = x(1+\alpha x)$ in the
  second case. 
\end{proof}

\begin{lem}
  \label{lem:full-action-on-conormal}
  Similarly, we may compute the right operator near $\lcfm$ away from
  the corner:
  \begin{align*}
    P(\lambda)^{t}_{R} u = &\sum _{j=0}^{\infty} \int
    e^{i\frac{\rho}{\st}\eta} (\eta + i0)^{\gamma_{j}+1} \left(
      -2i\hat{g}(d_{R}\rho, \st d_{R}\tilde{a}_{j}) +
      i(\pd[\st]\rho)(n + 3 + 2\gamma_{j} + O(\st)
      )\tilde{a}_{j}\right) \deta \, \nu \\ 
    &+ \sum _{j=0}^{\infty} e^{i\frac{\rho}{\st}\eta} (\eta + i0)
    ^{\gamma_{j}} \bigg( P(\lambda)_{R}^{t} +
    (\gamma_{j}+O(\st))\st\pd[\st] + (\gamma_{j}+1)(n+3+2\gamma_{j})
    \notag \\ 
    &\quad \left. + \frac{1}{2}\left(n-\frac{1}{2}\right) +
      O(\st)\st\pd[\zt] + O(\st)\right) \tilde{a}_{j}\deta\, \nu 
  \end{align*}

  We may also compute the behavior of the left and right operators
  near $\lcfp\cap \lcfm$:
  \begin{align*}
    P(\lambda)u \asympto &\sum_{j=0}^{\infty} \int
    e^{i\frac{\rho}{x\xt}\eta} (\eta + i0) ^{\gamma_{j} + 1}
    \frac{1}{\xt}\left(-2i\hat{g}(d\rho, xd\tilde{a}_{j}) -
      i(\pd[x]\rho)(n-3-2\gamma_{j}+O(x))\tilde{a}_{j}
    \right)\deta\,\nu \notag\\ 
    &+\sum_{j=0}^{\infty}\int e^{i\frac{\rho}{x\xt}\eta} (\eta + i0)
    ^{\gamma_{j}} \bigg( P(\lambda) + (\gamma_{j} + O(x))x\pd[x] -
    (\gamma_{j}+1)(n-3-2\gamma_{j}) \\ 
    &\quad \left.  + \frac{1}{2}\left(n-\frac{1}{2}\right) +
      O(x)x\pd[y] + O(x)\right)\tilde{a}_{j}\deta\,\nu \\ 
    P(\lambda)^{t}_{R} u = &\sum _{j=0}^{\infty} \int
    e^{i\frac{\rho}{x\xt}\eta} (\eta + i0)^{\gamma_{j}+1}
    \frac{1}{x}\left( -2i\hat{g}(d_{R}\rho, \xt d_{R}\tilde{a}_{j})
      +i(\pd[\xt]\rho)(n + 3 + 2\gamma_{j} + O(\xt)
      )\tilde{a}_{j}\right) \deta \, \nu \\ 
    &+ \sum _{j=0}^{\infty} e^{i\frac{\rho}{x\xt}\eta} (\eta + i0)
    ^{\gamma_{j}} \bigg( P(\lambda)_{R}^{t} +
    (\gamma_{j}+O(\xt))\xt\pd[\xt] + (\gamma_{j}+1)(n+3+2\gamma_{j})
    \notag \\ 
    &\quad \left. + \frac{1}{2}\left(n-\frac{1}{2}\right) +
      O(\xt)\xt\pd[\yt] + O(\xt)\right) \tilde{a}_{j}\deta\, \nu 
  \end{align*}
\end{lem}

Because we wish to solve $P(\lambda)u = r$ up to smooth terms, we wish
to iteratively solve away the terms in the above expansions. The first
transport equation we must solve is then 
\begin{equation*}
  -2\hat{g}(d\rho , sd \tilde{a}_{0}) - (\pd[s]\rho)(n-3-2\gamma_{0} +
  O(s))\tilde{a}_{0} = r_{0}, 
\end{equation*}
where $r_{0}$ is compactly supported and comes from the inhomogeneous
term.

Because $\pd[s]\rho$ is nonzero, we may divide by it to obtain the
transport equation 
\begin{equation*}
  -\frac{h^{kl}\left(\pd[z_{k}]\rho\right)
    s\pd[z_{l}]\tilde{a}_{0}}{\pd[s]\rho} + s\pd[s]\tilde{a}_{0} -
  (\frac{n}{2}-\frac{3}{2}-\gamma_{0} + O(s))\tilde{a}_{0} =
  \frac{r_{0}}{2\pd[s]\rho} 
\end{equation*}
Note that $\gamma_{0}=k + \frac{n}{2}-2$, so the coefficient of
$\tilde{a}_{0}$ is just $k-\frac{1}{2} + O(s)$.

Near the face $\lcfp$, given by $s=0$, we may use a parameter $t$
along the light cone $\lightcone$. The parameter $t$ is then equivalent to $s$, so
we may change coordinates to 
\begin{equation*}
  t\pd[t] \tilde{a}_{0} + (k-\frac{1}{2})\tilde{a}_{0} = O(t) \tilde{a}_{0} + \tilde{r}.
\end{equation*}
We note that the solution $\tilde{a}_{0}$ of this equation must have a
polyhomogeneous expansion in $t$. In fact, we prove a more precise
version of this statement.

\begin{lem}
  \label{lem:TE-sols-are-phg}
  Suppose that $v$ solves the differential equation
  \begin{equation*}
    t\pd[t] v - (j - k+\frac{1}{2}) v = t\cdot c(t)v + b,
  \end{equation*}
  where $c$ is smooth in $t$, and $b$ is polyhomogeneous in $t$ with
  index set
  \begin{equation*}
    \E _{j-1}=\left\{ \left(-k +\frac{1}{2}+ l, i\right) : l\in
      \naturals, i \leq l\text{ if }l \leq j-1, i=j-1\text{ if }l \geq
      j\right\} 
  \end{equation*}
  when $j\neq 0$ and $b$ is supported away from $t=0$ for $j=0$. Then
  $v$ has a polyhomogeneous expansion in $t$ with index set
  \begin{equation*}
    \E_{j}=\left\{ \left(-k +\frac{1}{2}+ l, i\right) : l\in \naturals, i \leq l\text{
      if }l \leq j, i=j\text{ if }l > j\right\}. 
  \end{equation*}
\end{lem}

\begin{note}
  We may prove a similar lemma for solutions of the transport
  equations on the right, with appropriate modifications for the index
  sets. \end{note}

\begin{proof}
  This lemma follows as a simple exercise in the b-calculus of Melrose
  (cf.\ \cite{grieser:2001} or \cite{melrose:1993}) or as an exercise
  in the theory of hyperbolic Fuchsian operators (cf.\
  \cite{MR701390},\cite{MR764343}, or \cite{MR935706}). Because the
  proof of this lemma is elementary, we include it here.

We show this by constructing a formal power series solution.

We start with the case $j=0$ so that near $t=0$, $b\equiv 0$. We seek
a formal power series solution of the form 
\begin{equation*}
  t^{-k+\frac{1}{2}}\sum _{l=0}^{\infty}v_{l}t^{l}.
\end{equation*}
Indeed, with this ansatz, the equation becomes
\begin{equation*}
  t^{-k+\frac{1}{2}}\sum _{l=0}^{\infty} l v_{l} = t^{-k+\frac{1}{2}}\sum _{l=0}^{\infty} q_{l-1},
\end{equation*}
where $q_{l-1}$ depends only on $c(t)$ and the first $l-1$
coefficients $v_{i}$ (so $q_{-1}=0$). The coefficient $v_{0}$ is fixed
by the initial condition of the differential equation, and then the
remaining coefficients may be found iteratively.

We then sum this series with Borel summation to find a function with
this power series at $t=0$ and the difference between this function
and the solution of the differential equation vanishes to all orders
at $t=0$.

For general $j$, we write 
\begin{equation*}
  b = \sum _{l=0}^{j-1}\sum _{i=0}^{\min(l,j-1)} b_{li} t^{l}(\log t)^{i}
\end{equation*}
with a similar expression for $v$.  A similar calculation then reduces
the equation to
\begin{align*}
  &\sum_{l=0}^{\infty}\sum _{i=0}^{\min(l,j)}t^{-k+\frac{1}{2}+l}(\log
  t)^{i}\left( (l-j)v_{li} + (i+1)v_{l,i+1}\right) \\ 
  &\quad = \sum _{l=1}^{\infty}\sum_{i=0}^{l}t^{-k+\frac{1}{2}+l}(\log
  t)^{i}\left(q_{l-1, i} + b_{li}\right), 
\end{align*} 
where $q_{l-1,i}$ depends on the function $c(t)$ and the coefficients
$v_{l'i'}$ where $l' \leq l-1$ or $l'=l$ and $i' > i$. In particular,
we may again iteratively solve for each coefficient and then Borel sum
the result. 
\end{proof}

We may now apply Lemma~\ref{lem:TE-sols-are-phg} to the first
transport equation 
\begin{equation*}
  t\pd[t]\tilde{a}_{0} - (j -k+\frac{1}{2})\tilde{a}_{0} + t c(t)\tilde{a}_{0} = b(t),
\end{equation*}
where $c(t)$ is smooth in $t$. We find that $a_{0}$ is polyhomogeneous
in $t$ with index set $\E_{0}$. Changing coordinates back to $s$ tells
us that $a_{0}$ is polyhomogeneous in $s$ with index set $\E_{0}$.

Letting $Q_{j}$ be the operator acting on $\tilde{a}_{j}$ in the
coefficient of $(\eta + i0)^{\gamma_{0}- j}$ in
equation~(\ref{eq:TE-op-to-ansatz-result}), the $j$th transport
equation is then 
\begin{equation*}
  -\frac{h^{kl}\left(\pd[z_{k}]\rho\right)s\pd[z_{l}]\tilde{a}_{j}}{\pd[s]\rho}
  + s\pd[s]\tilde{a}_{j} - (j-k+\frac{1}{2})\tilde{a}_{j} +
  O(s)\tilde{a}_{j} = -Q_{j}\tilde{a}_{j-1}. 
\end{equation*}
By applying Lemma~\ref{lem:TE-sols-are-phg} again, we may conclude
that $a_{j}$ is polyhomogeneous in $s$ with index set $\E_{j}$.

Now let $\E = \cup _{j}\E_{j}$. By repeating the process above, we
obtain conormal distributions $u_{j}$ in $\CDf[\E]{k-3/2 -
  j}{\dblspacet}{\lightcone}$ with symbols $a_{j}$ such that
\begin{equation*}
  P(\lambda)\left(\sum _{j=0}^{N}u_{j}\right)- r \in
  \CDf[\E]{k-1/2-N-1}{\dblspacet}{\lightcone}. 
\end{equation*}

We now wish to asymptotically sum this expression to find $u$ such
that $P(\lambda)u - f$ is smooth on the interior of $\dblspacet$.

\begin{lem}
  \label{lem:TE-asymp-sum}
  There is a distribution $u\in
  \CDf[\E]{k-3/2}{\dblspacet}{\lightcone}$, supported on the interior
  of the light cone, such that
  \begin{equation*}
    u \asympto \sum _{j=0}^{\infty}u_{j} .
  \end{equation*}
\end{lem}

\begin{proof}
  Each symbol $(\eta + i0)^{\gamma _{0} - j}a_{j}$ is analytic (in
  $\eta$) and exponentially decreasing in the upper half plane $\Im
  \eta > 0$, so the Paley-Wiener theorem tells us that each $u_{j}$ is
  supported in the region $\{ \frac{\rho}{s} \geq 0\}$. A standard
  Borel summation argument completes the proof.
\end{proof}

Putting our factors of $r_{\frontface}$ from
Section~\ref{sec:diagonal-singularity} back in, we have thus proved:
\begin{prop}
  Given $r\in r_{\frontface}^{p}\CDf[0]{k-1/2}{\dblzero}{\lightcone}$,
  there is a conormal distribution,
  \begin{equation*}
    u\in r_{\frontface}^{p}\phgalt{\F_{1}}\CDf{k-3/2}{\dblspacet}{\lightcone},
  \end{equation*}
  such that 
  \begin{equation*}
    P(\lambda)u - r \in r_{\frontface}^{p}\phg{\F_{1}}{\dblspacet}{\lcfp},
  \end{equation*}
  where $u$ is supported away from the corner $\ffp\cap \lfp$ and
  $\F_{1}$ is given by
  \begin{equation*}
    \F_{1} = \left\{ (j - k +\frac{1}{2}, l) : j,l\in \naturals, l\leq j\right\}.
  \end{equation*}
  We may further arrange that the distribution is supported in the
  interior of the light cone. 
\end{prop}

We observe here that an argument similar to the one given in the
previous section shows that $P_{R}^{t}u - r$ is again smooth on the
interior when we are constructing a parametrix for the fundamental
solution. In particular, the symbols of $u$ satisfy the transport
equations for the actions on the right factor, and so the following
proposition holds: 
\begin{prop}
  \label{prop:TE-conormal-error}
  There is a conormal distribution $u \in r_{\frontface}^{-n/2}
  \mathcal{A}^{\F}\CDf{-3/2}{\dblspacet}{\lightcone}$ such that
  \begin{align*}
    P(\lambda) u - r &\in r_{\frontface}^{-n/2} \phgalt{\E_{L}}(\dblspacet), \\
    P(\lambda) ^{t}_{R}u - r &\in r_{\frontface}^{-n/2}\phgalt{\E_{R}}(\dblspacet).
  \end{align*}
  Here the index families $\F$ and $\E$ are given by
  \begin{align*}
    F_{\lcfp} &= \left\{ (j  +\frac{1}{2}, l) : j,l\in \naturals, l\leq j \right\} \\
    F_{\lcfm} &= \left\{ (j + \frac{1}{2} - n, l) : j, l\in \naturals, l\leq j \right\} \\
    E_{\lcfp,L} &= \left\{ (j +\frac{1}{2}, l) : j,l\in \naturals, l\leq j\right\} \\
    E_{\lcfm,L} &=\left\{ (j + \frac{1}{2} - n -1, l) : j, l\in \naturals, l\leq j\right\}  \\
    E_{\lcfp,R} &= \left\{ (j +\frac{1}{2} - 1, l) : j,l\in \naturals, l\leq j\right\} \\
    E_{\lcfm,R} &= \left\{ (j + \frac{1}{2} - n, l) : j, l\in \naturals, l\leq j\right\}.
  \end{align*}
\end{prop}

Note that the decrease of $-1$ for two of the error terms in the above
proposition come from the factors of $\frac{1}{x}$ in the expression
given in Lemma~\ref{lem:full-action-on-conormal}. The $-n$ on the
``minus'' faces comes from taking the transpose of our operator. This
is because we are using sections of the standard half-density bundle
rather than the $0$-half-density bundle.

\section{The light cone face}
\label{sec:light-cone-face}

We now wish to solve away the error from
Proposition~\ref{prop:TE-conormal-error}. We aim to solve away the
left error at $\lcfp$ and the right error at $\lcfm$. We show only the
left calculation here and observe that the right calculation is nearly
identical.

We again call the error $r$, which is in
$r_{\frontface}^{p}\phgalt{\E_{L}}(\dblspacet)$, i.e., near $\ffp$, $r$
has an expansion of the form 
\begin{equation*}
  r \asympto \xt^{p}\sum _{j}^{\infty}\sum _{l=0}^{j}s^{\gamma_{j}}(\log s)^{l}r_{jl}\nu,
\end{equation*}
where $\gamma_{j} = j - k - \frac{1}{4}$. We drop the power of $\xt$
for now because $P(\lambda)$ commutes with $\xt$.

We first claim that this error lifts to be polyhomogeneous on the full
double space $\dblspace$. 
\begin{lem}
  Suppose that $r$ is a polyhomogeneous distribution on $\dblspacet$
  supported in a small neighborhood of $\lightcone$ with index family
  $\F$. Then $r$ lifts to a polyhomogeneous distribution on
  $\dblspace$ with index family $\G$, where $\G$ is given by
  \begin{align*}
    G_{\lcfp} &= F_{\lcfp}, \\
    G_{\lcfm} &= F_{\lcfm}, \\
    G_{\xf} &= F_{\lcfp} + F_{\lcfm}+1.
  \end{align*}
\end{lem}

\begin{note}
  Here the notation $F + 1$ is shorthand for
  \begin{equation*}
    (\alpha , l) \in F + 1 \text{ if and only if } (\alpha - 1, l) \in F.
  \end{equation*}
\end{note}

\begin{proof}
  The result follows because the two possible orders of the blow-up
  are locally diffeomorphic near $\lightcone \cap \lcfp\cap \lcfm$.
  The extra $1$ in $G_{\xf}$ is because sections of $\hd(\dblspacet)$
  lift to sections of $r_{\xf}\hd(\dblspace)$. 
\end{proof}

We may thus consider $r$ as a polyhomogeneous function on $\dblspace$.

We now proceed in two steps. The first is to solve the away near
$\lcfp$ and the second is to show that it has the desired form at
$\lfp$. This only away from the scattering face $\xf$, though the
computation near $\xf$ is nearly identical. In this section there are
many terms that come from differentiating our ansatz. We attempt to
indicate the origin of the important terms.

Note that the statement about the support in
Proposition~\ref{prop:TE-conormal-error} means that $r_{jl}$ is
supported on the interior of the light cone and vanishes to infinite
order at $\lightcone$.

Because we are working near $\lcfp$, we first use projective
coordinates $(s,w = \frac{\rho}{s}, \theta)$, where $\rho$ is a
defining function for $\lightcone$ (as above) and $\theta$ are the
remaining variables. In these coordinates, derivatives of the function
$\rho$ appear as coefficients of $\pd[w]$ in our operator
$P(\lambda)$. In particular, $s\pd[s]$ lifts to $s\pd[s] +
(\pd[s]\rho)\pd[w] - w\pd[w]$.

We again expect a polyhomogeneous expansion in $s$. In other words, we
expect an expansion of the form 
\begin{equation}
  \label{eq:LCF-ansatz}
  u \asympto \sum _{j=0}^{\infty}\sum _{l=0}^{j}s^{\gamma_{j}}(\log s)^{l}u_{jl},
\end{equation}
where $u_{jl}$ is regarded as a function of $w$, $\theta$, $\xt$, and
$\yt$.

Write $u_{jl}=\tilde{u}_{jl}\nu$. We again note that we may again take
$r_{\lcfp}$ equivalent to $s$ in this region, giving us an extra
$-\frac{1}{2}s\pd[s] - \frac{1}{2}(n-\frac{1}{2})$ in our operator.
Applying $P(\lambda)$ to our ansatz yields 
\begin{align*}
  P(\lambda) u \asympto &\sum_{j} \sum_{l=0}^{j} s^{\gamma_{j}}(\log
  s)^{l}\left[ -\bar{g}(d\rho, d\rho) - 2w (\pd[s]\rho) +
    w^{2})(\pd[w]^{2}\tilde{u}_{jl}) \right. \\ 
  &\quad - \left((n+\frac{1}{2}-2\gamma_{j})(\pd[s]\rho) -
    (n+\frac{1}{2})w\right)\pd[w]\tilde{u}_{jl} \notag\\ 
  &\quad - \left( (n-1-\gamma_{j})\gamma_{j}+
    \frac{1}{2}(n-\frac{1}{2})-\lambda\right)\tilde{u}_{jl} \notag \\ 
  &\quad \left. + A_{j-2}\tilde{u}_{j-2,l} + B_{j-1}\tilde{u}_{j-1,l}
    + B_{l+1}'\tilde{u}_{j-1,l+1} + X_{jl}\tilde{u}_{j,l+1} +
    Y_{l}\tilde{u}_{j,l+2}\right]\nu, \notag 
\end{align*}
where
\begin{align*}
  A_{j+2} &= -\sum _{q}(\lap _{h}\theta _{q})\pd[\theta_{q}] - \sum
  _{q,r,i,k}\left( \frac{\pd \theta_{r}}{\pd
      z_{i}}\right)\left(\frac{\pd\theta_{q}}{\pd
      z_{k}}\right)\pd[\theta_{r}]\pd[\theta_{q}], \quad B_{l+1}' =
  \frac{\xt \pd[s]\sqrt{h}}{\sqrt{h}}(l+1),\\ 
  B_{j+1} &= -\sum_{i,k,r}2h^{ik}\left(\frac{\pd\theta_{r}}{\pd z_{i}}\right) \left( \frac{\pd\rho}{\pd z_{k}}\right) \pd[\theta_{r}]\pd[w] + (\lap _{h}\rho)\pd[w] + (\pd[s]^{2}\rho)\pd[w], \\
  X_{jl} &= 2(l+1)(\pd[s]\rho - w)\pd[w] -
  (l+1)(n-\frac{1}{2}-2\gamma_{j}), \quad Y_{l} = (l+1)(l+2). 
\end{align*}

The constants above come from the operator and from $s\pd[s]$ landing
on the powers of $s$. The leading $\pd[w]$ terms come from the
$(\pd[s]\rho)\pd[w]$ terms when we lift $s\pd[s]$ and the
$h^{kl}(\pd[z_{k}]\rho)(\pd[z_{l}]\rho)\pd[w]^{2}$ term in the lift of
the Laplacian in $z$.

Note that because $\lightcone$ is characteristic for $\Box$, we must
have that $\bar{g}(d\rho, d\rho) = \rho b$. $\rho = sw$, so by
replacing $B_{j+1}$ with $\tilde{B}_{j+1} = B_{j+1} - wb\pd[w]^{2}$,
we may write 
\begin{align*}
  P(\lambda) u \asympto &\sum_{j} \sum_{l=0}^{j} s^{\gamma_{j}}(\log
  s)^{l}\left[ Q_{j}\tilde{u}_{jl} + A_{j-2}\tilde{u}_{j-2,l} +
    \tilde{B}_{j-1}\tilde{u}_{j-1,l} \right.\\ 
  &\quad \left.+ B_{l+1}'\tilde{u}_{j-1,l+1} + X_{jl}\tilde{u}_{j,l+1}
    + Y_{l}\tilde{u}_{j,l+2}\right]\nu, \notag 
\end{align*}
where
\begin{align*}
  Q_{j} = &\left( -(\pd[s]\rho) + w\right)w\pd[w]^{2} - \left(
    (\pd[s]\rho)(n+\frac{1}{2}-2\gamma_{j}) -
    (n+\frac{1}{2})w\right)\pd[w] \\ 
  &- (n-1-\gamma_{j})\gamma_{j} - \frac{1}{2}(n-\frac{1}{2})-\lambda . 
\end{align*}

We thus wish to solve a sequence of inhomogeneous ordinary
differential equations $Q_{j}\tilde{u}_{jl} = \tilde{r}_{jl}$, where
\begin{equation*}
  \tilde{r_{jl}} = r_{jl} - X_{jl}\tilde{u}_{j,l+1} -
  Y_{l}\tilde{u}_{j,l+2} - A_{j-2}\tilde{u}_{j-2,l} -
  \tilde{B}_{j-1}\tilde{u}_{j-1,l} - B_{l+1}'\tilde{u}_{j-1,l+1}. 
\end{equation*}
where all terms are supported in $\{ w \geq 0\}$, vanishing to
infinite order at $w=0$.  (We know already that $r_{jl}$ has this
property, and we show at each step that $\tilde{u}_{jl}$ does as
well.)

It is clear that $0$ is a regular singular point of the differential
operator $Q_{j}$, so the solutions of $Q_{j}v = 0$ have formal power
series expansions at $w=0$ with first term given by $w^{\mu_{i}}$.
Here $\mu_{i}$ are the roots of the indicial equation $-\mu (\mu -1) -
(n+\frac{1}{2}-2\gamma_{j})\mu =0$ (see, for example,
\cite{birkhoff-rota:1989}), i.e.,
\begin{equation*}
  \mu _{1} = 0, \quad \mu _{2} = 2j-2k-n
\end{equation*}
Here we have used that $\gamma_{j} = j-k+\frac{1}{2}$.  Standard ODE
techniques (i.e., variation of parameters) then give us solutions
$\tilde{u}_{jl}$ to $Q_{j}\tilde{u}_{jl} = \tilde{r}_{jl}$ in terms of
a basis of solutions for $Q_{j} v = 0$.  Moreover, because
$\tilde{r}_{jl}$ vanishes to all orders at $w=0$, we may also
guarantee that $\tilde{u}_{jl}$ is supported in $\{ w \geq 0\}$,
vanishing to all orders at $w=0$.  Indeed, if $v_{1}$ and $v_{2}$ are
a basis for the solutions of $Q_{j}v=0$, then
\begin{equation*}
  \tilde{u}_{jl}(w) = -v_{1}(w)\int
  _{0}^{w}\frac{v_{2}(w')\tilde{r}_{jl}(w')}{W(v_{1},v_{2})(w')}\dw' +
  v_{2}(w)\int_{0}^{w}\frac{v_{1}(w')\tilde{r}_{jl}(w')}{W(v_{1},v_{2})(w')}\dw' .
\end{equation*}
Because $\tilde{r}_{jl}$ vanishes to all orders at $w=0$, (and $v_{i}$
are bounded by $w^{-N}$ for some $N$), the integrals make sense and
vanish to all orders at $w=0$. Multiplication by $v_{i}$ then
preserves this property.

We may thus solve these equations and now wish to consider their
asymptotics near the corner $\lcfp\cap \lfp$. Near this corner, the
coordinates $(s,w, \theta)$ are invalid and so we must use the other
set of projective coordinates $(\rho = sw, W = w^{-1}, \theta)$. In
these coordinates, $s\pd[s]$ lifts to $(\pd[s]\rho)\pd[\rho] + W\pd[W]
- (\pd[s]\rho)W^{2}\pd[W]$. The top order terms below then come from
the $W\pd[W]$ and the constants come only from the operator.

Because polyhomogeneous distributions are independent of our choice of
coordinate systems, we may also express our ansatz
(\ref{eq:LCF-ansatz}) in terms of these coordinates $(\rho, W,
\theta)$. In this case, expansions in $s$ are equivalent to expansions
in $\rho$, so our ansatz here has the form 
\begin{equation*}
  u \asympto \sum _{j}\sum _{l=0}^{j} \rho ^{\gamma_{j}}(\log \rho)^{\ell}v_{jl}.
\end{equation*}
In the computation that follows, the important point is that
$s_{\pm}(\lambda)$ are the indicial roots of our operator
$P(\lambda)$, and that this behavior is dominant away from
$\lightcone$ because the fundamental solution is a smooth solution of
the homogeneous equation here.

Apply $P(\lambda)$ to such an ansatz to see
\begin{align*}
  P(\lambda) u \asympto &\sum_{(j,l)\in \F_{1}} \rho ^{j}(\log
  \rho)^{l}\left[(1 - q_{2}W)W^{2}\pd[W]^{2} \tilde{v}_{jl} \right. \\ 
  &\quad - (n-2 + q_{1}W)W\pd[W]\tilde{v}_{jl} - (\lambda +
  q_{0}W)\tilde{v}_{jl} \notag \\ 
  &\quad + WA_{j-1,l}\tilde{v}_{j-1,l} +
  WA_{j-1,l+1}'\tilde{v}_{j-1,l+1} \notag \\ 
  &\quad \left. +WB_{j-2,l}\tilde{v}_{j-2,l} +
    WB_{j-2,l+1}'\tilde{v}_{j-2,l+1} +
    WB_{j-2,l+2}''\tilde{v}_{j-2,l+2}\right], \notag 
\end{align*}
where we have (via a similar, but more tedious, calculation)
\begin{align*}
  q_{2} &= 2(\pd[s]\rho) - (\pd[s]\rho)^{2}W + |d_{z}\rho|^{2}_{h}W, \\
  q_{1} &= -(n+2\gamma_{j} - 4)(\pd[s]\rho) +
  (2\gamma_{j}-2)\bar{g}(d\rho, d\rho)W, \\ 
  q_{0} &= (n-2)\gamma_{j}(\pd[s]\rho) -
  \gamma_{j}(\gamma_{j}-1)\bar{g}(d\rho, d\rho)W, 
\end{align*}
while 
\begin{align*}
  &A_{j-1,l} = 2h^{ik}(\pd[z_{i}]\theta
  _{p})(\pd[z_{k}]\rho)W^{2}\pd[\theta_{p}]\pd[W] -
  2\gamma_{j}h^{ik}(\pd[z_{i}]\rho)(\pd[z_{k}]\theta_{p})W\pd[\theta_{p}]
  \\ 
  &\quad \quad \quad - (\Box _{\bar{g}}\rho)W^{2}\pd[W] -
  \gamma_{j}(\Box_{\bar{g}}\rho)W, \\ 
  &A_{j-1,l+1}'= (l+1)\big[\left( -2(\pd[s]\rho) + 2\bar{g}(d\rho,
    d\rho)\right)W\pd[W] \\ 
  &\quad\quad\quad - (\pd[s]\rho)(n-2) - (2\gamma_{j}-2)\bar{g}(d\rho,
  d\rho)W\big] , 
\end{align*}
and
\begin{align*}
  &B_{j-2,l}= (\lap _{h}\theta _{p})W\pd[\theta _{p}] -
  h^{ik}(\pd[z_{i}]\theta _{p})(\pd[z_{k}]\theta _{q})W\pd[\theta
  _{p}]\pd[\theta_{q}] \\ 
  &B_{j-2,l+1}' =
  -2(l+1)h^{ik}(\pd[z_{k}]\rho)(\pd[z_{i}]\theta_{p})W\pd[\theta_{p}]
  + (l+1)(\Box_{\bar{g}}\rho)W \\ 
  &B_{j-2,l+2}'' = -(l+1)(l+2)\bar{g}(d\rho, d\rho)W. 
\end{align*}

We must still take into account the density factor here. In this
region, $r_{\lcfp}$ is equivalent to $\rho$. A similar computation to
the one above shows that this leads to two types of terms. One of
these is $O(W)W\pd[W]$, while the other is $O(W)$. This means that
they may be absorbed into $q_{1}$ and $q_{0}$. Let us now write
\begin{equation*}
  Q_{j} = (1 - q_{2}W)W^{2}\pd[W]^{2} - (n-2+q_{1}W)W\pd[W] - (\lambda
  + q_{0}W). 
\end{equation*}
We wish to solve a sequence of transport equations (which are really
the same equations as above written in these coordinates) given by
\begin{align*}
  Q_{j}\tilde{v}_{jl} &= \tilde{r}_{jl}- WA_{j-1,l}\tilde{v}_{j-1,l} -
  WA_{j-1,l+1}'\tilde{v}_{j-1,l+1} -WB_{j-2,l}\tilde{v}_{j-2,l} \\ 
  &\quad - WB_{j-2,l+1}'\tilde{v}_{j-2,l+1} -
  WB_{j-2,l+2}''\tilde{v}_{j-2,l+2} = \tilde{r}_{jl}'. 
\end{align*}

The solutions of these transport equations are polyhomogeneous conormal
functions. Here the indicial roots of $Q_{j}$ are $s_{\pm}(\lambda) =
\frac{n-1}{2} \pm \sqrt{\frac{(n-1)^{2}}{4}+\lambda}$. More precisely,
we have 
\begin{lem}
  \label{lem:LCF-transport-eqn-solution-is-phg}
  Suppose that $u_{jl}$ solves $Q_{j}u_{jl} = \tilde{r}_{jl}'$ as
  above. Suppose first that $s_{+}(\lambda) - s_{-}(\lambda) \notin
  \integers$. Suppose that $\tilde{r}_{jl}'$ is polyhomogeneous
  conormal at $\lfp$ with index set $F + 1 = (F^{+} + 1) \cup
  (F^{-}+1)$, where
  \begin{equation*}
    F^{\pm} = \{ (s_{\pm}(\lambda) + m, 0 ) : m\in \naturals , \} ,
  \end{equation*}
  then $u_{jl}$ is polyhomogeneous with index set $F = F^{+}\cup
  F^{-}$.
  
  If $s_{+} (\lambda) - s_{-}(\lambda) = N\in \integers$, then if
  $\tilde{r}_{jl}'$ is polyhomogeneous with index set $\tilde{F} + 1$,
  where
  \begin{equation*}
    \tilde{F} = \{ (s_{-}(\lambda) + m , l ) : m\in \naturals , l=0
    \text{ for }l<N, l=1 \text{ for } l\geq N\}, 
  \end{equation*}
  then $u_{jl}$ must be polyhomogeneous with index set $\tilde{\F}$.
\end{lem}

\begin{note}
  Note that the index set above means that if $s_{+}(\lambda) -
  s_{-}(\lambda) \notin \integers$, then $u_{jl}\in
  W^{s_{+}(\lambda)}C^{\infty} + W^{s_{-}(\lambda)}C^{\infty}$. If
  $s_{+}(\lambda) - s_{-}(\lambda) \in \integers$, then
  \begin{equation*}
    u_{jl} - \sum _{m=0}^{s_{+}(\lambda) - s_{-}(\lambda)
      -1}W^{s_{-}(\lambda) + m}u_{jl}^{(m)} \in (1+\log W) C^{\infty}. 
  \end{equation*}
\end{note}

\begin{proof}
  We may again construct a formal power series solution (or apply
  Lemma~5.44 of Melrose \cite{melrose:1993}). We omit this here
  because it has been described in detail already. The key point is
  that the operators $A_{j-1,l}$, $A_{j-1,l+1}'$, $B_{j-2,l}$,
  $B_{j-2,l+1}'$, and $B_{j-2,l+2}''$ are all elements of
  $W\Diff[*]_{b}$ (products of $W$ and differential operators tangent
  to the boundary).  This ensures that the terms we are solving away
  vanish to one order better at $W= 0$.
\end{proof}

Asymptotically summing $\sum \sum \rho ^{\nu _{j}}(\log \rho
)^{l}u_{jl}$ then solves away the error $r$ at $\lcfp$. Note that
because $s_{\pm}(\lambda)$ are the indicial roots of $N(P(\lambda))$,
we in fact have that the error term vanishes to one order better.

When we are constructing a parametrix for the fundamental solution, we
may perform the same construction on the right and the left. We may
also add a distribution solving away the right error at $\lcfm$.

Remembering our factors of $r_{\frontface}$ now, we have now proved
the following proposition. 
\begin{prop}
  \label{prop:LCF-solution-at-newface}
  Given $r_{1}\in r_{\frontface}^{-n/2}\phgalt{\E_{L}}(\dblspace)$ and
  $r_{2}\in r_{\frontface}^{-n/2}\phgalt{E_{R}}(\dblspace)$ as above,
  we may find a smooth function $u\in
  r_{\frontface}^{-n/2}\phgalt{\F}(\dblspace)$ vanishing outside the
  light cone such that $P(\lambda)u - r_{1}$ vanishes to all orders at
  $\lcfp$ and is polyhomogeneous with index family $\G_{1}$, while
  $P(\lambda)^{t}_{R}u - r_{2}$ vanishes to all orders at $\lcfm$ and
  is polyhomogeneous with index set $\G_{2}$. Here we have that
  \begin{align}
    \label{eq:LCF-full-index-set}
    &F_{\lfp} = G_{\lfp,2} = \{ (s_{\pm}(\lambda) + m , 0 ) : m\in
    \naturals \},& &F_{\lcfp} = E_{\lcfp,L},\\ 
    &F_{\rfp} = G_{\rfp, 1}= \{ (-n + s_{\pm}(\lambda) + m, 0) : m\in
    \naturals\},& &F_{\lcfm} = E_{\lcfm,R}, \notag\\ 
    &G_{\lcfm,1} = F_{\lcfm} - 1, & & G_{\lfp,1} = F_{\lfp} + 1,
    \notag\\ 
    &G_{\lcfp,2} = F_{\lcfp} - 1, & & G_{\rfp,2} = F_{\rfp} + 1,
    \notag\\ 
    &G_{\lcfp, 1} = G_{\lcfm, 2} = E_{\lfp,L} = E_{\rfm, L} =
    E_{\lfp,R} = E_{\rfp,R} = \emptyset, & & \notag\\ 
    &F_{\xf} = E_{\xf, L} = E_{\xf, R} = G_{\xf,1} = G_{\xf, 2} = \{
    (-n+j, l) : l \leq j\}, & & \notag 
  \end{align}
  if $s_{+}(\lambda) - s_{-}(\lambda) \notin \integers$. If $s_{+} -
  s_{-}(\lambda) = N \in \integers$, then the index sets become
  \begin{align*}
    &\tilde{F}_{\lfp} = \{ (s_{-}(\lambda) + m , l ) : m \in
    \naturals, l = 0 \text{ for }m< N, l= 1\text{ for }m \geq N\} ,\\ 
    &\tilde{F}_{\rfp} = \{ (-n+s_{-}(\lambda) + m, l ) : m \in
    \naturals, l = 0 \text{ for }m< N, l= 1\text{ for }m \geq N\} , 
  \end{align*}
  with corresponding changes for $G_{\lfp}$ and $G_{\rfp}$.
\end{prop}

\section{The front face}
\label{sec:front-face}

We now wish to solve away the errors on the front face, which we again
call $r$, from the previous step. We show how to solve away the error
term at $\ffp$ for the operator acting on the left, and the
corresponding calculation at $\ffm$ for the operator acting on the
right is nearly identical. We now suppose that $r$ is the error term
from Proposition~\ref{prop:LCF-solution-at-newface} for the operator
acting on the left.

Because $r$ vanishes to all orders at $\lcfp$, we may blow down
$\lcfp$ to solve away $r$. In this view, $\xt^{-p}r$ is smooth on
$\ffp$, supported inside the light cone, and has an expansion at
$\lfp$ of the form 
\begin{equation*}
  r\asympto \xt^{p}\sum _{(\alpha _{j},l)\in \G} s^{\alpha _{j}}(\log s )^{l}r_{jl}\mu.
\end{equation*}
The generic case here is that $l=0$ when $s_{+}(\lambda) -
s_{-}(\lambda) \notin \integers$. We again drop the powers of $\xt$
from our notation.

We wish to solve this error away with a function of the same form:
\begin{equation*}
  u \asympto \sum _{(\alpha_{j},l)\in \mathcal{G}}s^{\alpha_{j}}(\log s)^{l}u_{jl}\mu .
\end{equation*}
Applying $N(P(\lambda))$ to such an ansatz yields
\begin{align}
  \label{eq:FF-apply-to-ansatz}
  N(P(\lambda)) u \asympto &\sum _{j}\sum _{l}s^{\alpha_{j}}(\log
  s)^{l}\left[ \alpha_{j}^{2}u_{jl} + 2\alpha_{j}(l+1)u_{j,l+1} +
    (l+1)(l+2)u_{j,l+2} \right. \\ 
  &\quad \left.- (n-1)\alpha_{j}u_{jl} - (n-1)(l+1)u_{j,l+1} + \lap
    _{z}u_{j-2,l} - \lambda u_{jl}\right] \nu\notag\\ 
  = &\sum_{j} \sum_{l} s^{\alpha_{j}}(\log s)^{l}\left[
    -\left((n-1-\alpha_{j})\alpha_{j} + \lambda\right) u_{jl} - (n-1
    -2\alpha_{j})(l+1)u_{j,l+1} \right. \notag\\ 
  &\quad \left.+ (l+1)(l+2)u_{j,l+2} + \lap _{z}u_{j-2,l} \right]\mu
  \notag. 
\end{align}
The coefficient of $u_{jl}$ in this expression vanishes precisely when
$\alpha_{j} = s_{\pm}(\lambda)$, so we may solve away the error to all
orders at $\lfp$ because the expansions of $r$ begin at
$s_{\pm}(\lambda) + 1$.

\begin{prop}
  \label{prop:FF-solve-at-leftface}
  For $s_{+}(\lambda) - s_{-}(\lambda) \notin \integers$, there is a
  smooth function $u$ on $\ffp$, polyhomogeneous at $\lfp$ with index
  set $F_{\lfp}$, where $F_{\lfp}$ is defined in
  equation~(\ref{eq:LCF-full-index-set}), such that $N(P(\lambda)) u -
  \xt^{-p}r$ vanishes to all orders at $\lfp\cap \ffp$. An identical
  statement holds when $s_{+}(\lambda)-s_{-}(\lambda)\in \integers$,
  with $F_{\lfp}$ replaced by $\tilde{F_{\lfp}}$. 
\end{prop}

\begin{proof}
  This is clear from the expression (\ref{eq:FF-apply-to-ansatz}) for
  $N(P(\lambda))u$. 
\end{proof}

On the front face, we are then left with an error term $r_{0}$
vanishing to all orders at $\lfp\cap \ffp$ and supported inside the
light cone. Because $\ffp$ is an asymptotically de Sitter space, we
may now use Corollary~3.6 of Vasy (\cite{vasy:2007}) to find a smooth
function on $\ffp$, vanishing to all orders at $\lfp$, such that
$N(P(\lambda))u = \xt^{-p}r_{0}$.

By iterating this construction (and extending it to the interior of
$\lfp$ as well), we may find $u\in
\xt^{p}\phg{\F_{2}}{\dblzero}{\lfp}$ such that $P(\lambda)u - r$
vanishes to all orders at $\ffp$ and $\lfp$. By gluing two of these
functions together, we may simultaneously solve the left error near
the ``plus'' faces and the right error near the ``minus'' faces. This
proves the following proposition. 
\begin{prop}
  \label{prop:ff-param-at-ff-and-lf} 
  We may solve away the left error
  at $\ffp$ and $\lfp$, and the right error at $\rfm$ and $\ffm$ with
  a function $u \in \phgalt{\mathcal{H}}(\dblspace)$, where
  $\mathcal{H}$ is given by 
  \begin{align*}
    H_{\lcfp} = H_{\lcfm}= \emptyset, 
    H_{*} = F_{*} \text{ for the other index sets}. 
  \end{align*} 
  The remaining error terms are in $\phgalt{\K_{L,R}}(\dblspace)$,
  where $\K_{L}$ and $\K_{R}$ are given by 
  \begin{align*}
    &K_{\lfp, L}= K_{\rfm, R} = K_{\lcfp} = K_{\lcfm} = \emptyset, & &K_{\lfp,R} = F_{\lfp},\\
    &K_{\rfm,L} = F_{\rfm}, & & K_{\xf, L} = K_{\xf,R} = H_{\xf}.
  \end{align*}
\end{prop}

\section{The full parametrix}
\label{sec:full-parametrix}

We now take the various pieces of the parametrix constructed in
Sections \ref{sec:diagonal-singularity}, \ref{sec:transport-equation},
\ref{sec:light-cone-face}, and \ref{sec:front-face} to construct a
parametrix for the fundamental solution of $P$.

Putting together the results of
Propositions~\ref{prop:PL-construction}, \ref{prop:TE-conormal-error},
\ref{prop:LCF-solution-at-newface}, and
\ref{prop:ff-param-at-ff-and-lf}, we have proved the following
theorem:

\begin{thm}
  \label{thm:thm1}
  Suppose that $X$ is an asymptotically de Sitter space, satisfying
  assumptions (A1), (A2), and (A3). We may find a left parametrix $K$
  such that $P(\lambda) K = I + R_{1}$ and $KP(\lambda) = I + R_{2}$,
  where the Schwartz kernels of $R_{1}$ and $R_{2}$ are smooth on the
  interior of $\dblspace$ and are polyhomogeneous with index families
  $\E_{L}$ and $\E_{R}$ on $\dblspace$. We may write $K = K_{1} +
  K_{2} + K_{3}$, where $K_{1}$ is supported in a neighborhood of the
  diagonal, $K_{2}$ is supported in a small neighborhood of the light
  cone $\lightcone$ away from the diagonal, and all three pieces are
  supported on the interior of the light cone. Moreover,
  \begin{align}
    K_{1} &\in r_{\frontface} ^{-n/2} \PLzero[-3/2], \\
    K_{2} &\in
    r_{\frontface}^{-n/2}\phgalt{\F}\CDf{-3/2}{\dblspacet}{\lightcone},
    \notag\\ 
    K_{3} &\in r_{\frontface}^{-n/2} \phgalt{\F}(\dblspace).\notag
  \end{align}
  The index families $\E_{L}$, $\E_{R}$, and $\F$ are given by
  \begin{align*}
    &F_{\lcfp} = \{ (j+\frac{1}{2}, l) : l \leq j, j \in \naturals\},
    & & F_{\lcfm} = F_{\lcfp} - n, \\ 
    &F_{\lfp} = E_{\lfp,R} = \{ (s_{\pm}(\lambda) + m , 0 ) : m\in
    \naturals \},& & E_{\rfm, L}= F_{\rfm} = F_{\lfp} - n, & & \\ 
    &E_{\lcfm,L} = F_{\lcfm} - 1, & & E_{\lcfp,R} = F_{\lcfp} - 1,
    \notag\\ 
    &F_{\xf} = E_{\xf, L} = E_{\xf, R} = \{ (-n+j, l) : l \leq j\}, &
    & \notag \\ 
    &E_{\ffm,L} = E_{\ffp,R}= \left\{ \left(-\frac{n}{2} + m,
        0\right) : m \in \naturals\right\}, & & \\ 
    &E_{\lcfp, L} = E_{\lcfm, R} = E_{\lfp, L} = E_{\rfm, R} =
    E_{\ffp,L} = E_{\ffm, R} =  \emptyset. & & \notag 
  \end{align*}
  If $s_{+}(\lambda) - s_{-}(\lambda) \in \integers$, then we must
  modify $F_{\lfp}$ and $F_{\rfp}$ (and the index sets depending on
  them) as described earlier. 
\end{thm}

\begin{note}
  As observed earlier, this theorem holds without the assumption (A3)
  as long as we are willing to multiply our distribution by a cutoff
  function supported in a neighborhood of the front face. In this
  case, the extra blow-up to obtain the scattering face $\xf$ is
  unnecessary. 
\end{note}

Because our remainder terms lose one order of decay at the light cone
faces, we lose an order of decay there when we pass to the exact
fundamental solution. The following is a precise statement of the main
result (Theorem~\ref{thm:main-thm}). 
\begin{thm}
  \label{thm:fund-soln}
  The exact forward fundamental solution $E_{+}$ is in this class of
  distributions, but with index sets $F_{\lcfp}' = F_{\lcfp}- 1$ and
  $F_{\lcfm}' = F_{\lcfm} - 1$. In other words, we may write $E_{+} =
  K_{1}+K_{2}+K_{3}$ with 
  \begin{align}
    K_{1} &\in r_{\frontface} ^{-n/2} \PLzero[-3/2], \\
    K_{2} &\in r_{\frontface}^{-n/2}\phgalt{\F}\CDf{-3/2}{\dblspacet}{\lightcone}, \notag\\
    K_{3} &\in r_{\frontface}^{-n/2} \phgalt{\F}(\dblspace).\notag
  \end{align}
  Here the index family $\F$ is given by
  \begin{align*}
    &F_{\lcfp} = \{ (j-\frac{1}{2}, l) : l \leq j, j \in \naturals\},
    & & F_{\lcfm} = F_{\lcfp} - n, \\ 
    &F_{\lfp} = \{ (s_{\pm}(\lambda) + m , 0 ) : m\in \naturals \},& &
    F_{\rfm} = F_{\lfp} - n, & & \\ 
    &F_{\xf} =  \{ (-n-1+j, l) : l \leq j\}, & & \notag  
  \end{align*}
  with modifications to $F_{\lfp}$ and $F_{\rfm}$ when $s_{+}(\lambda)
  - s_{-}(\lambda)$ is an integer. 
\end{thm}

\begin{note}
  If we instead adopt the convention that the $0$-densities are flat,
  that
  \begin{equation*}
    \int _{X}f(\xt, \yt) \delta(x-\xt)\delta(y-\yt) \dg (x,y) = f(x,y),
  \end{equation*}
  and we write $K$ as a section of the pullback bundle of
  ${}^{0}\hd(X\times X)$, then the index sets change somewhat. Indeed, if
  $K= \hat{K}\tilde{\nu}$, where $\tilde{\nu}$ is a nonvanishing
  section of the pullback of ${}^{0}\hd (X\times X)$, then the same
  theorem holds, but with index sets
  \begin{align*}
    F_{\lcfp} &= F_{\lcfm}= \{ (j-1,l) : l\leq j, j\in \naturals\}, \\
    F_{\lfp} &= F_{\rfm} = \{ (s_{\pm}(\lambda)+m,0): m\in\naturals\}, \\
    F_{\xf} &= \{ (j-1, l): l\leq j, j\in \naturals\}, \\
    F_{\ffp} &= F_{\ffm} = \{ (j, 0) : j\in \naturals\}.
  \end{align*}
\end{note}

\begin{proof}[Proof of theorem]
  $E_{+}$ is the forward fundamental solution, so if $f$ is a
  compactly supported smooth function on $X$, then $P(\lambda)E_{+}f =
  f$ and $E_{+}P(\lambda)f = f$. Moreover, continuity allows us to
  extend this to any forward-oriented distribution. In particular, if
  $f$ is any smooth function on the interior of $X$ vanishing to all
  orders at $Y_{-}$ that is also a tempered distribution on $X$, then
  \begin{align*}
    P(\lambda)E_{+}f &= f \\
    E_{+}P(\lambda) f &= f.
  \end{align*}
  
  Let $K$ be the parametrix for $E_{+}$ constructed in
  Theorem~\ref{thm:thm1}. If $f$ is a smooth function on $X$,
  vanishing to all orders at $Y$, then $Kf$ vanishes to all orders at
  $Y_{-}$ because $K$ is identically zero in a neighborhood of $\lfm$
  and the lift of $f$ vanishes to all orders at $\lcfm$, $\ffm$,
  $\xf$, and $\rfm$. A similar argument applies to $R_{1}f$ and
  $R_{2}f$. We may then write
  \begin{align*}
    (Kf) &= E_{+}P(\lambda)Kf = E_{+}f + E_{+}R_{1}f \\
    (Kf) &= KP(\lambda)E_{+}f = E_{+}f + R_{2}E_{+}f .
  \end{align*}
  In particular,
  \begin{equation*}
    E_{+} = K - KR_{1} + R_{2}E_{+}R_{1}.
  \end{equation*}
  We then observe that the error terms $KR_{1}$ and $R_{2}E_{+}R_{1}$
  have the desired properties, finishing the proof. 
\end{proof}

\section{Modifications for de {S}itter space}
\label{sec:exact-de-sitter}

De Sitter space does not satisfy assumption (A3) because the
projection of the forward flowout of the light cone from a point at
past infinity intersects itself at future infinity (though not on the
interior of the spacetime). We briefly discuss the modifications
to our construction needed for de Sitter space. The most important
modification is to solve the transport equations on $\dblspace$ rather
than $\dblspacet$. (Indeed, we could have done this from the outset,
but chose not to because solutions of the transport equations are
easier to understand on $\dblspacet$.)

Observe that the construction detailed above works without difficulty
away from the corner $\lfp \cap \rfm$. As mentioned earlier,
$\lightcone$ intersects itself in this corner. This intersection is
given by $\{ (0,y,0,-y): y\in Y\}$. Near this intersection, we may
write $\rho = x + \xt - |y+\yt|$, plus terms vanishing to higher order
at the boundary. Note that the function $\rho$ is no longer smooth at
this intersection because $|y+\yt| = 0$ there.

We resolve this singularity by blowing up the submanifold
$\lightcone\cap \lfp\cap\rfm$ as in
Definition~\ref{defin:ABU-full-double-space}. Although this
submanifold had codimension $3$ when $X$ satisfied (A3), it has
codimension $n+1$ here. Indeed, this blow-up is almost the same as the
one in Section~\ref{sec:0-geometry} that defined $\dblzero$. The
function $\rho$ now lifts to be smooth on this new space $[\dblzero ,
\lightcone\cap\lfp\cap\rfm]$. Indeed, after the blow-up we may write
\begin{equation}
  \label{eq:rho-on-ds}
  \rho = (1+s) - |z| + O(\xt)  
\end{equation}
near $\xf \cap\lfp$. This is now smooth near $\rho = 0$, and so we may
blow up $\lightcone\cap \lfp$ and $\lightcone \cap \rfm$ to obtain
$\dblspace$.

We must also modify the manner in which we solve the transport
equations. The behavior at $\lcfp$ and $\lcfm$ may be obtained in the
same way as in Section~\ref{sec:transport-equation}, but the behavior
at $\xf$ requires a slightly different approach. In
Section~\ref{sec:another-blow-up}, we showed that the symbol of the
conormal distribution was polyhomogeneous at $\xf$ by constructing it
on $\dblspacet$ and lifting it to $\dblspace$. Because $\rho$ is not
smooth on $\dblzero$ when $X$ is the de Sitter space, we cannot solve
the transport equations on the intermediate double space $\dblspacet$
up to the corner and instead we must solve the transport equation
along $\lcfm$ and $\lcfp$.

Solving the transport equation on these faces requires using the
semi-explicit form of $\rho$. The terms where the operator lands
entirely on $\rho$ in equation~(\ref{eq:TE-op-to-ansatz-result}) can
no longer be ignored. Using the form (\ref{eq:rho-on-ds}) for $\rho$,
we observe that these now contribute a constant term to the equation.
Because $\Box \rho = s - (n-1)s + s^{2}(n-1) + O(\xt)$ and $\pd[s]\rho
= 1 + O(\xt)$,

A computation in the same spirit as those in
Section~\ref{sec:transport-equation} shows that the symbol of $K_{2}$
is also polyhomogeneous at $\xf$ with index set 
\begin{equation*}
  F_{\xf} = \left\{ \left( -n + m, 0\right): m\in \naturals\right\}.
\end{equation*}
Note that this is the same index set we found before.

The rest of the construction proceeds without change.

The modifications to the construction for de Sitter space correspond
to allowing the location of the pole $p$ in $P(\lambda) u =
\delta_{p}$ tend to past infinity. If we require that the point $p$ is
uniformly bounded away from past infinity, no modification is
necessary.

In \cite{polarski:1989} and \cite{yagdjian-galstian:2008}, the authors
do not consider sending this pole to past infinity. Our unmodified
construction recovers a slightly weaker version of the results of
these authors when the pole is in the interior but the modified
construction also describes the behavior of the fundamental solution
when the pole is at past infinity.

\section{Polyhomogeneity}
\label{sec:mapping-props-phg}

The aim of this section is to prove Theorem~\ref{thm:mapping-phg},
which was stated in Section~\ref{sec:introduction}.

We begin by considering the maps $\beta_{L}$ and $\beta _{R}$, where
$\beta _{L,R}$ are given by composing the blow-down maps with
projections onto each factor, as in the diagram here: 
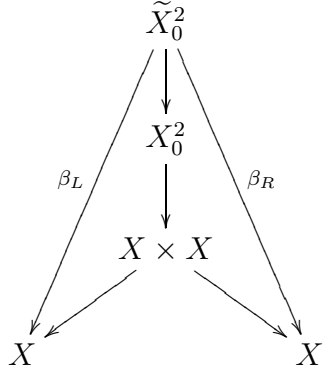
\begin{equation*}
  \xymatrix{
    & \dblspace \ar[d]\ar[dddl]_{\beta_{L}}\ar[dddr]^{\beta _{R}}& \\
    & \dblzero \ar[d]& \\
    & X \times X \ar[dl] \ar[dr] & \\
    X & & X
  }
\end{equation*}

We require the following four lemmata.

\begin{lem}
  \label{lem:phg-maps-are-b-fibrations}
  The maps $\beta _{L}$ and $\beta _{L}|_{\lightcone^{int}}$ are b-fibrations.
\end{lem}

\begin{lem}
  \label{lem:phg-transverse-fibers}
  The fibers of $\beta_{L}$ and $\beta _{L}: \lightcone^{int} \to X$
  are transverse to $\lightcone$. 
\end{lem}

\begin{lem}
  \label{lem:phg-vfs-lift}
  For $V$ a b-vector field on $X$, there is a $b$-vector field
  $\tilde{V}$ on $\dblspace$ such that $(\beta _{L})_{*}\tilde{V} = V$
  and $\tilde{V}$ is tangent to $\lightcone$. 
\end{lem}

\begin{lem}
  \label{lem:phg-model-form-of-conormal-bit}
  Suppose that $M$ and $N$ are manifolds with corners, $F:M\to N$ is a
  b-fibration, $H$ is a boundary hypersurface of $M$, and $F(H) = N$
  Suppose also that $K\in \CDf{0}{M}{H}$ is polyhomogeneous at the
  other boundary hypersurfaces of $M$ (satisfying the hypotheses of
  Lemma~\ref{lem:phg-b-fibration-pushforward}). Then $F_{*}K$ is a
  polyhomogeneous distribution on $N$. 
\end{lem}

\begin{proof}[Proof of Lemma~\ref{lem:phg-maps-are-b-fibrations}]
  Here $\lightcone^{int}$ is taken as an open manifold away from the
  cone edge near the diagonal. $K_{2}$ is supported away from the cone
  point, so we may restrict our attention to this region.

  That $\beta _{L}$ and $\beta _{L}|_{\lightcone^{int}}$ are
  b-fibrations (defined in Section~\ref{sec:polyhomogeneity}) follows
  from a more general statement: if $F : M \to N$ is a b-fibration,
  and $Z\subset \pd M$ is a p-submanifold, then $F \circ \beta : [M ;
  Z] \to N$ is a b-fibration. In particular, the blow-down map is a
  composition of b-fibrations. 
\end{proof}

\begin{proof}[Proof of Lemma~\ref{lem:phg-transverse-fibers}]
  Because $\Lambda_{1}$ is the flowout Lagrangian in the right factor,
  its intersection with the fibers $\xt s=x_{0}$, $\yt + \xt z =
  y_{0}$ is the flowout of the light cone with cone point at
  $(x_{0},y_{0})$. In particular, this is an embedded submanifold of
  $X$, and so the intersection is transverse. 
\end{proof}

\begin{proof}[Proof of Lemma~\ref{lem:phg-vfs-lift}]
  This is really a consequence of
  Lemma~\ref{lem:phg-transverse-fibers}. Indeed, we may choose a local
  basis of vector fields given by $\pd[\nu]$ and $V_{j}$, where
  $\pd[\nu]$ is tangent to the fibers of $\beta _{L}$ and $V_{j}$ are
  tangent to $\lightcone$. Then $\left( \beta
    _{L}\right)_{*}\pd[\nu]=0$ because $\pd[\nu]$ is tangent to the
  fibers of $\beta _{L}$, and so we may choose a lift of the vector
  fields so that the $\pd[\nu]$ component vanishes at $\lightcone$.
\end{proof}

\begin{proof}[Proof of Lemma~\ref{lem:phg-model-form-of-conormal-bit}]
  We choose $K^{\epsilon}$ polyhomogeneous on $M$, supported away from
  $H$ such that $K^{\epsilon}$ are uniformly bounded in
  $\CDf{0}{M}{H}$ and converge to $K$ in $\CDf{\delta}{M}{H}$ for any
  $\delta > 0$. The pushforward theorem of Melrose
  (Lemma~\ref{lem:phg-b-fibration-pushforward} of this paper) tells us that
  $F_{*}K^{\epsilon}$ are polyhomogeneous with fixed index set.

  We claim now that the $F_{*}K^{\epsilon}$'s are Cauchy and so converge to
  $F_{*}K$. This guarantees that $F_{*}K$ is polyhomogeneous. The key
  observation here is that if $\tilde{V}$ is a lift of $V$, then
  \begin{equation*}
    V\left( F_{*}\left(K^{\epsilon_{1}}- K^{\epsilon _{2}}\right)
    \right)  = F_{*}\left( \tilde{V} \left( K^{\epsilon _{1}} -
        K^{\epsilon _{2}} \right) \right) . 
  \end{equation*}

  The expression $\tilde{V} \left(K^{\epsilon _{1}}- K_{\epsilon
      _{2}}\right)$ tends to $0$ because $K^{\epsilon}\to K$ as
  distributions conormal to $H$. 
\end{proof}

\begin{proof}[Proof of Theorem~\ref{thm:mapping-phg}]
  A simple wavefront set argument gives us that the $K_{2}f$ piece is
  smooth.

  Thinking functorially, we write 
  \begin{equation*}
    K(v\gamma) = (\beta _{L})_{*} \left( K \cdot \beta
      _{R}^{*}(v\gamma)\right) = u\gamma . 
  \end{equation*}
  In particular, if we write $K = \tilde{K}\nu$, then
  \begin{equation*}
    u\gamma ^{2} = u \left| \dx\dy\right| = (\beta_{L})_{*}\left(
      \tilde{K}
      (\beta_{R}^{*}v)r_{\frontface}^{n/2}(r_{\lcfp}r_{\lcfm})^{1/2}r_{\xf}\nu^{2}\right). 
  \end{equation*}

  We now use the decomposition $K = K_{1} + K_{2}+K_{3}$, where
  $K_{i}$ are as in Theorem~\ref{thm:fund-soln}. The lemma above shows
  that $\beta _{L}$ is a b-fibration, and so we may use the
  pushforward theorem of Melrose in
  Lemma~\ref{lem:phg-b-fibration-pushforward} to treat the
  contribution from $K_{3}$. In particular, we note that if $v$ has
  index set $E$ at $Y_{+}$ and vanishes to all orders at $Y_{-}$, then
  $\beta _{R}^{*}v$ is polyhomogeneous on $\dblspace$, smooth at
  $\lcfp$ and $\leftface$, with index set $E$ at $\ffp$, and vanishing
  to all orders at $\ffm$, $\lcfm$, and $\xf$ because it is
  forward-directed. Index sets add when functions are multiplied, and
  so we know that
  $r_{\frontface}^{n/2}(r_{\lcfp}r_{\lcfm})^{1/2}r_{\xf}\tilde{K}_{3}\cdot
  \beta _{R}^{*}f$ has index family $\G$ given by
  \begin{align*}
    &G_{\ffp} = E, & &G_{\ffm} = \emptyset\\
    &G_{\lfp} = \{ (s_{\pm}(\lambda) + m, 0) : m \in \naturals\}, & &G_{\lfm} = \emptyset\\
    &G_{\rfm} = \emptyset, & &G_{\rfp} = \emptyset\\
    &G_{\lcfp} = \{ (j,l): l \leq j, j \in \naturals\}, & &G_{\lcfm} = \emptyset\\
    &G_{\xf} = \emptyset.
  \end{align*}
  $G_{\lfp}$ must be modified when $s_{+}(\lambda) - s_{-}(\lambda)$
  is an integer. We may now use Lemma
  \ref{lem:phg-b-fibration-pushforward} to conclude that $K_{3}f$ is
  polyhomogeneous on $X$ with index set
  \begin{equation*}
    E \,\extcup\, G_{\lfp} \, \extcup \, G_{\lcfp}.
  \end{equation*}

  To handle the contribution from $K_{2}$, we use the lemmas above.
  Indeed,
  Lemmas~\ref{lem:phg-maps-are-b-fibrations},~\ref{lem:phg-vfs-lift},
  and \ref{lem:phg-model-form-of-conormal-bit} show that $K_{2}f$ is
  polyhomogeneous with index set $G_{\lcfp}\, \extcup\, E$.

  This leaves only the contribution from $K_{1}$. This is just a
  consequence of the local theory of paired Lagrangian distributions.
  Indeed, the work of Joshi in \cite{joshi:1998} implies that the
  pushforward of $K_{1}\beta _{R}^{*}f$ exists as a smooth function.
  The uniformity of $K_{1}$ down to the front face $\frontface$ then
  tells us that $K_{1}f$ is polyhomogeneous with index set $E$.
\end{proof}

\section{An $L^{p}$ estimate}
\label{sec:lp-estimate}

As another application of Theorem~\ref{thm:thm1}, we consider the
behavior of the $L^{p}$ norms of a family of smooth compactly
supported functions with support tending towards $Y_{+}$.

Suppose first that $R\in \Psi ^{-\infty}_{0}(X)$ is a smoothing
$0$-pseudodifferential operator in the small calculus of
\cite{mazzeo:1988} or \cite{mazzeo-melrose:1987}, supported near
$\ffp$. In other words, the Schwartz kernel of $R$ is a smooth
function on $\dblzero$, supported away from $\leftface$ and
$\rightface$ and near $\ffp$. Concretely, let $\phi$ is a smooth,
compactly supported function on $\reals^{n}_{+}$, supported near
$(1,0)$, and with $\int_{0}^{\infty}\int _{\reals^{n-1}} \phi
\frac{\ds\dz}{s^{n}}= 1$. Let $\chi$ be a smooth function on $X$ that
is identically zero near $Y_{-}$. If $R(s,z,\xt,\yt) = \phi (s,z)\chi
(\xt,\yt)$, then $R\in \Psi^{-\infty}_{0}(X)$ is such an operator.

Define now a family of compactly supported functions $f$ given by
\begin{equation*}
  f_{(\xt, \yt)} (x,y) = R\left(\frac{x}{\xt}, \frac{y-\yt}{\xt}, \xt, \yt\right) = Rv,
\end{equation*}
where $v = \delta (x' - \xt) \delta (y' - \yt)$. Each $f_{(\xt, \yt)}$
is a smooth function on $X$ supported in a compact neighborhood of
$(\xt, \yt)$ with unit $L^{1}(X;\differential{g})$-norm.

Because $f_{(\xt,\yt)}$ is given by applying $R$ to a $\delta$
function, pointwise bounds for $E_{+}f_{(\xt, \yt)}$ are equivalent to
pointwise bounds on the Schwartz kernel of $E_{+}R$.

The following lemma is useful for obtaining pointwise bounds.
\begin{lem}
  \label{lem:ptwise-estimates-dblspace}
  Suppose that $K = \hat{K}\nu$, where $K$ is a polyhomogeneous
  function on $\dblspace$ with index family $\F$ and supported near
  $\ffp$. Suppose that $\F$ satisfies
  \begin{align*}
    F_{\ffp} &= \{ (s_{\frontface}+j, p) : j \in \naturals, p \leq p_{j}\}, \\
    F_{\lcfp} &= \{ (s_{\newface}+j, p) : j \in \naturals, p \leq p_{j}\}, \\
    F_{\lfp} &= \{ (s_{\leftface}+j, p) : j \in \naturals, p \leq p_{j}\}, \\
    F_{\rfp} &= \{ (s_{\rightface}+j,p) : j\in \naturals, p\leq p_{j}\},
  \end{align*}
  and $p_{0} = 0$ for each index set (i.e., no log terms appear in the
  top order part of the expansion). Suppose further that
  \begin{align*}
    s_{\frontface} \geq \frac{n}{2}, \quad   s_{\newface} \geq \frac{1}{2}, \quad s_{\leftface} \geq 0, \quad s_{\rightface} \geq 0.
  \end{align*}
  Then $K$ (considered as an operator) satisfies
  \begin{equation*}
    \norm[L^{\infty}(X)]{Ku} \leq C\norm[L^{1}(X; \dg)]{u}.
  \end{equation*}

  If $s_{\frontface} \geq -\frac{n}{2}$ instead, then $K$ is bounded
  $L^{1}(X; \differential{g}) \to L^{\infty}(X)$. 
\end{lem}

\begin{proof}
  The half-density $\left| \dg_{L}\dg_{R}\right|^{1/2}$ lifts to a
  nonvanishing smooth multiple of $r_{\lcfp}^{1/2}r_{\ffp}^{n/2}\nu$
  near $\ffp$, so we may write
  \begin{equation*}
    \hat{K}\nu = \hat{K}r_{\lcfp}^{-1/2}r_{\ffp}^{-n/2}
    \beta^{*}\left(\left| \dg_{L}\dg_{R}\right|^{1/2}\right), 
  \end{equation*}
  where $\beta^{*}$ is the pullback by the blow-down map $\dblspace
  \to X\times X$. The assumptions above mean that
  $\hat{K}r_{\lcfp}^{-1/2}r_{\ffp}^{-n/2}$ is a bounded function on
  $\dblspace$ and so is the pullback of a bounded function on $X\times
  X$.

  The second statement of the lemma follows from the observation that
  $\dg = x^{n}\differential{g}$. 
\end{proof}

Let $K_{1},K_{2}$, and $K_{3}$ be the decomposition of $E_{+}$ given
in Theorem~\ref{thm:thm1}. The Schwartz kernel of $K_{j}R$ is given by
\begin{equation}
  \label{eq:composition}
  \int K_{j}\left(\frac{x}{s'\xt}, \frac{y - (\yt + \xt z')}{s'\xt},
    s'\xt, \yt + \xt z'\right) R(s',z',\xt,\yt)\ds'\dz' 
\end{equation}
where $s' = \frac{x'}{\xt}$ and $z' = \frac{y'-\yt}{\xt}$.  This
corresponds to writing the composition of operators $A$ and $B$ on
$X\times X$ as 
\begin{equation*}
  \kappa_{AB}(x,y,\xt,\yt) = \int
  \kappa_{A}(x,y,x',y')\kappa_{B}(x',y',\xt,\yt)\dx'\dy', 
\end{equation*}
where $\kappa_{A}$ here denotes the Schwartz kernel of $A$.  If
$\kappa_{A}$ and $\kappa_{B}$ are instead functions of $s,z,\xt,\yt$,
where $s = \frac{x}{\xt}$ and $z = \frac{y-\yt}{\xt}$, this becomes 
\begin{align*}
  &\int \kappa_{A}\left(\frac{x}{x'}, \frac{y-y'}{x'}, x',
    y'\right)\kappa_{B} \left(\frac{x'}{\xt},
    \frac{y'-\yt}{\xt},\xt,\yt\right)\dx'\dy' \\ 
  &\quad = \int \kappa _{A}\left(\frac{x}{s'\xt}, \frac{y - (\yt + \xt
      z')}{s'\xt}, s'\xt, \yt + \xt
    z'\right)\kappa_{B}\left(s',z',\xt,\yt\right)\ds'\dz', 
\end{align*}
which yields equation~(\ref{eq:composition}).

Note that $K_{1},K_{2}$, and $K_{3}$ all vanish identically near
$\rfp$. Because the family $f_{(\xt,\yt)}$ is supported away from
$Y_{-}$, we may assume that $K_{i}$ are supported away from the
``minus'' faces $\ffm,\lcfm,\rfm$, and $\xf$.

Consider first $K_{1}$, the piece corresponding to the paired
Lagrangian singularity of $E_{+}$. By Theorem~\ref{thm:thm1}, the
Schwartz kernel of $K_{1}$ is given by $(x')^{-n/2}\tilde{K}_{1}$,
where $\tilde{K}_{1}$ is a paired Lagrangian distribution on
$\dblzero$ supported away from $\leftface$ and $\rightface$.

Because the fibers of integration in equation~(\ref{eq:composition})
are transverse to the diagonal and to the light cone, it follows that
$K_{1}R = \xt ^{-n/2}u_{1}$, where $u_{1}$ is a smooth function on
$\dblzero$ supported away from $\leftface$ and $\rightface$. In
particular, $u_{1}$ is bounded, and
Lemma~\ref{lem:ptwise-estimates-dblspace} implies that
\begin{equation*}
  \norm[L^{\infty}]{K_{1}f_{(\xt, \yt)}} \leq C.
\end{equation*}

Because $R$ is defined only near $\ffp$, we are free to use the
improved parametrix of Theorem~\ref{thm:thm1} rather than distribution
of Theorem~\ref{thm:fund-soln}. In particular, the symbol of the
conormal distribution $K_{2}$ is bounded by $r_{\lcfp}^{1/2}$ and so
$K_{2}R$ satisfies the conditions of the second part of
Lemma~\ref{lem:ptwise-estimates-dblspace} and so
\begin{equation*}
  \norm[L^{\infty}(X)]{K_{2}f_{(\xt,\yt)}}\leq C.
\end{equation*}

Consider finally the polyhomogeneous term $K_{3}$. The Schwartz kernel
of $r_{\lfp}^{-s_{-}(\lambda)}K_{3}R$ satisfies the conditions of the
lemma, and so 
\begin{equation*}
  \norm[x^{s_{-}(\lambda)}L^{\infty}(X)]{K_{3}f_{(\xt,\yt)}}\leq C.
\end{equation*}

Now let $l = \max \left( 0, -\Re s_{-}(\lambda)\right)$. Putting the
estimates for $K_{1},K_{2}$, and $K_{3}$ together yields
\begin{equation*}
  \norm[x^{-l}L^{\infty}(X)]{E_{+}f_{(\xt,\yt)}} \leq C.
\end{equation*}

An $L^{2}$ estimate for this family follows from
Theorem~\ref{thm:Vasy-thm-5.4} (from \cite{vasy:2007}). By reversing
the roles of $Y_{-}$ and $Y_{+}$ in this theorem, we conclude that for
any forward-directed $f\in x^{-r}L^{2}(X)$ and $r > \max\left(
  \frac{1}{2}, l(\lambda)\right)$, there is a unique $u\in
x^{-r}H^{1}_{0}(X)$ such that $P(\lambda)u = f$, and 
\begin{equation*}
  \norm[x^{-r}L^{2}(X;\differential{g})]{u} \leq
  \norm[x^{-r}H^{1}_{0}(X)]{u}\leq C
  \norm[x^{-r}L^{2}(X;\differential{g})]{f}. 
\end{equation*}
Here $H^{1}_{0}$ is the $0$-Sobolev space of order one, i.e., it
measures regularity with respect to the $x\pd[x]$ and $x\pd[y]$ vector
fields.

In order to apply this estimate to the family $f_{(\xt,\yt)}$, it is
important to understand how the $L^{2}$-norms of the functions vary.
Indeed, a simple calculation shows that 
\begin{equation*}
  \norm[x^{-r}L^{2}(X)]{f_{(\xt,\yt)}}^{2} = \int
  _{X} |f_{(\xt,\yt)}(x,y)|^{2} \frac{\dx\dy}{x^{n+r}} = \int |\phi
  (s,z)|^{2}\frac{\ds\dz}{s^{n+r}\xt^{r}} = \xt^{-r}C_{r}^{2}, 
\end{equation*}
where $C_{r}$ depends on $r$, but not on $\xt$ or $\yt$. In
particular, $\norm[x^{-r}L^{2}(X;\differential{g})]{f} =
\xt^{-r/2}C_{r}$.

We may now prove Theorem~\ref{thm:Lp-estimates}.
\begin{proof}
  Interpolating between the $L^{\infty}$ and $L^{2}$ estimates
  provides an $L^{p}$ estimate for $p\in (2,\infty)$. Indeed, if $r >
  \max \left(\frac{1}{2}, \Re
    \sqrt{\frac{(n-1)^{2}}{4}+\lambda}\right)$, $l = \max \left(0,
    -\frac{n-1}{2} + \Re\sqrt{\frac{(n-1)^{2}}{4}+\lambda}\right)$,
  and $\frac{1}{p} = \frac{\theta}{2}$, $\theta \in [0,1]$, then
  \begin{equation}
    \label{eq:Lp-est-big}
    \norm[x^{-r\theta -
      l(1-\theta)}L^{2/\theta}(X;\differential{g})]{E_{+}f_{(\xt,\yt)}}\leq
    C\xt ^{-r\theta}, 
  \end{equation}
  which finishes the proof.
\end{proof}

\begin{note}
  The proof of Theorem~\ref{thm:Lp-estimates} uses the inclusion
  $H^{1}_{0} \subset L^{2}$. In particular, we ignore one derivative
  of $E_{+}f$ and so we could modify equation~(\ref{eq:Lp-est-big}) to
  include a fractional derivative. 
\end{note}

\section{Acknowledgements}
\label{sec:acknowledgements}

The author is very grateful to Rafe Mazzeo and Andr{\'a}s Vasy for
countless helpful conversations and to MSRI for their generous
hospitality while writing a draft of this paper. This research was
partly supported by NSF grants DMS-0805529 and DMS-0801226.

\bibliographystyle{plain}
\bibliography{math}

\end{document}